%% file: main.tex
\documentclass[11pt,letterpaper]{article}

 \usepackage{booktabs}       % professional-quality tables
\usepackage{tcs}
\usepackage{mysymbols} 
\usepackage{mathdots}
\usepackage{cleveref}
\def\showauthornotes{1}
\ifnum\showauthornotes=1
\newcommand{\anote}[1]{{\sf\color{orange}{ [Ainesh: #1] }}}
\newcommand{\gnote}[1]{{\sf\color{teal}{ [Goutham: #1] }}}

\else
\newcommand{\anote}[1]{}
\newcommand{\gnote}[1]{}
\fi

\newcommand{\norm}[1]{\ensuremath{\lVert #1 \rVert}}
\newcommand{\EE}{\mathbb{E}}

\title{A New Approach to Learning Linear Dynamical Systems}

\author{
Ainesh Bakshi\thanks{Supported by Ankur Moitra's ONR grant} \\
\texttt{ainesh@mit.edu} \\
MIT
\and
Allen Liu\thanks{Supported by an NSF Graduate Research Fellowship and a Fannie and John Hertz Foundation Fellowship} \\
\texttt{cliu568@mit.edu} \\
MIT
\and
Ankur Moitra\thanks{Supported by a grant from the ONR and a David and Lucile Packard Fellowship.} \\
\texttt{moitra@mit.edu} \\
MIT
\and
Morris Yau \\
\texttt{morrisy@mit.edu} \\
MIT
}

\date{}
\begin{document}

\maketitle
\begin{abstract}
    Linear dynamical systems are the foundational statistical model upon which control theory is built. Both the celebrated Kalman filter and the linear quadratic regulator require knowledge of the system dynamics to provide analytic guarantees.  Naturally, learning the dynamics of a linear dynamical system from linear measurements has been intensively studied since Rudolph Kalman's pioneering work in the 1960's \cite{kalman1960new}.  Towards these ends, we provide the first polynomial time algorithm for learning a linear dynamical system from a polynomial length trajectory up to polynomial error in the system parameters under essentially minimal assumptions; observability, controllability, and marginal stability.  Our algorithm is built on a method of moments estimator to directly estimate Markov parameters from which the dynamics can be extracted.  Furthermore we provide statistical lower bounds when our observability and controllability assumptions are violated.                        
\end{abstract}

\thispagestyle{empty}

\clearpage
\newpage

\microtypesetup{protrusion=false}
\tableofcontents{}
\thispagestyle{empty}
\microtypesetup{protrusion=true}

\clearpage

\setcounter{page}{1}

\section{Introduction}

Linear dynamical systems are the canonical model for time series data. At each time step $t$ there is an unknown hidden state $x_t \in \mathbb{R}^n$ and a known exogenous input $u_t \in \mathbb{R}^p$. The transition dynamics and observations $y_t \in \mathbb{R}^m$ are generated according to the following rules:
\begin{equation*}
\begin{split}
    x_{t+1} &= A x_{t} + B u_t + w_t ,\\
    y_{t} &= C x_t + D u_t + z_t ,
\end{split}
\end{equation*}
Here $A$, $B$, $C$ and $D$ are matrices of dimension $n \times n$, $n \times p$, $m \times n$ and $m \times p$ respectively. Moreover $w_t$ and $z_t$ are independent random variables and are called the process and observation noise respectively. It is standard to assume that they, along with the inputs and the initial state $x_0$, are all Gaussian, though we will work in a more general setting. 

Linear dynamical systems have wide-ranging applications in control theory \cite{grewal2010applications}, computer vision \cite{doretto2003dynamic}, speech recognition \cite{mesot2007switching}, econometrics \cite{athans1974importance}, healthcare \cite{li2007robust} and neuroscience \cite{schiff2009kalman}. They are the de facto model of choice due to their mathematical simplicity and because, when the parameters are known, making predictions about subsequent observations and making inferences about the unknown state are both algorithmically tractable. In fact these algorithms are simple, practical and statistically optimal. 

\emph{But what happens when the parameters are unknown?} The problem of estimating $A$, $B$, $C$ and $D$ from input-output sequences is called system identification and has been intensively studied since Rudolph Kalman's pioneering work in the 1960's \cite{kalman1960new}. There is a well-developed theory that furnishes asymptotic guarantees \cite{aastrom1971system, ljung1998system}. And more recently, many researchers have sought finite-sample guarantees both in the fully observed setting where $C = I$ \cite{faradonbeh2018finite, dean2020sample, simchowitz2018learning, sarkar2019near} and in the partially observed setting \cite{hardt2018gradient, oymak2019non, tsiamis2019finite, sarkar2019nonparametric, simchowitz2019learning}. Our focus here will be on obtaining running time and sample complexity bounds that are polynomial in the appropriate parameters and work under the most general conditions.

\subsection{Previous Work}

In the fully observed setting, the maximum likelihood estimator can be computed by solving ordinary least squares. It is known to be statistically optimal and there are strong finite sample guarantees on its performance \cite{faradonbeh2018finite, dean2020sample, simchowitz2018learning, sarkar2019near}. The partially observed setting is significantly more challenging because the problem of computing the maximum likelihood estimator becomes nonconvex. The EM algorithm \cite{ghahramani1996parameter} is often used in practice but it can get stuck in bad local minima. Our main focus will be on algorithms for learning partially observed linear dynamical systems with provable guarantees. There is a vast literature on this and related prediction problems (see Section~\ref{sec:related-work}). But all existing algorithms need to make one or more of the following types of restrictive assumptions:

\begin{itemize}
    \item[(1)] \textbf{Assumptions about the characteristic polynomial $q$ of $A$ or the phases of its roots.} Hardt, Ma and Recht \cite{hardt2018gradient} assumed that the image of the complex unit disk under $q$ is contained in the cone of complex numbers whose real part is larger than the absolute value of its imaginary part. For example, this is satisfied if the $\ell_1$-norm of the coefficients of $q$ is at most $\sqrt{2}/2$. Hazan et al. \cite{hazan2018spectral} studied the problem of predicting subsequent observations in a non-stochastic setting. Their bounds depend on the $\ell_1$-norm of the coefficients of a polynomial $p$ that vanishes on the phases of the eigenvalues of $A$. In particular, when there are few distinct roots or they are pairwise separated, the $\ell_1$-norm of the coefficients of $p$ can be much smaller than for $q$. This notion was further refined by Simchowitz et al. \cite{simchowitz2019learning}. However it is not clear why one would expect these norm bounds to be small. In many settings, unless there is extreme cancellation, the coefficients of $q$ would in fact be exponentially large. 
    
    \item[(2)] \textbf{Strict stability and mixing.} Another popular assumption is called strict stability, which stipulates that the spectral radius $\rho(A) < 1$. Often the transition matrix $A$ only satisfies $\rho(A) \leq 1$, which is called marginal stability. Consider a classic application in control theory, of tracking an object from radar measurements. The state of the object at some time step is its position, velocity and acceleration. The transition matrix is derived from Newton's laws and is upper triangular with ones along the diagonal, and so all of its eigenvalues are one. There are many other such examples, particularly in econometrics and coming from discretizations of ODEs. Algorithms that assume strict stability generally have bounds that depend on $1/(1-\rho(A))$ \cite{shah2012linear, hardt2018gradient}. Essentially, strict stability requires that the distribution of the $y_t$'s eventually converges and that there are no long-range correlations. So after about $1/(1-\rho(A))$ steps we essentially get fresh independent samples. Yet in many applications long-range correlations are an essential feature of the problem. Moreover getting around strict stability has many qualitative parallels with learning in graphical models without correlation decay \cite{bresler2015efficiently}, and learning in Gaussian graphical models without the restricted eigenvalue condition \cite{kelner2020learning}. 
    
    \item[(3)] \textbf{Restrictions on the dimension, etc.} Some algorithms only work in the single-input single-output setting, i.e. when $m = p =1$ \cite{hardt2018gradient}. Others have bounds that depend exponentially on the size of the largest Jordan block of $A$, or even treat the number of parameters of the linear dynamical system as a constant \cite{simchowitz2019learning}. 
\end{itemize}

By now, there is a standard blueprint which works as follows: The first step is to estimate the 
Markov parameters, given by
$$\begin{bmatrix} D &  CB &  CAB &  \cdots&  CA^s B \end{bmatrix}$$
The second step is to apply the Ho-Kalman algorithm \cite{ho1966effective}, which uses the Markov parameters to compute estimates $\hat{A}$, $\hat{B}$, $\hat{C}$ and $\hat{D}$ that are close to the true parameters in the appropriate metric. Oymak and Ozay \cite{oymak2019non} gave the first effective stability bounds for the Ho-Kalman algorithm. Thus the main issue is: How do you estimate the Markov parameters? Essentially all previous works use some form of linear regression. The analysis is based on expressing the observation $y_t$ as a linear function of the previous inputs and some noise terms. The Markov parameters can then be extracted from the regressor. The noise terms are a function of observation and process noise and also the quantity $A^s x_{t-s}$, which captures how the state at some previous time step affects the current state. When $A$ is strictly stable, this term decays exponentially. But when $A$ is only marginally stable, controlling this error presents many challenges. 

Our main question is:
\begin{quote}
\centering 
    \textit{Are there efficient algorithms for learning high-dimensional linear dynamical systems whose running time and sample complexity are polynomial in the appropriate parameters, and whose assumptions are essentially optimal? } 
\end{quote}
%Our goal is to work under essentially minimal assumptions: observability, controllability and marginal stability. Moreover we will prove statistical lower bounds for the setting when our assumptions are violated.

\subsection{Our Assumptions}
\label{subec:our-assumptions-informal}
%Next, we describe the assumptions we make on the the linear dynamical system itself, as well as our relaxations on the distribution of the control input and noise. 

It is important to draw a sharp distinction between the assumptions featured in the previous subsection and the more standard assumptions from control theory. In 1960, Rudolph Kalman \cite{kalman1960general} introduced the concepts of observability and controllability. Since then, it has been understood that they ought to in some sense govern what sorts of linear dynamical systems can be learned. In this subsection, we will review these assumptions and their natural quantitative counterparts. 

\paragraph{Observability and Controllability.}

%We begin by ensuring that the linear dynamical system at hand is not degenerate, i.e. we want to rule out the matrix $A$ only spanning a low-dimensional subspace of $\mathbb{R}^n$ and the matrix $C$ projecting away from this subspace. If that were to happen, the observations, $y_t$ would only contain the noise and it would be information-theoretically impossible to learn $A$. 
Consider the \emph{observability matrix}: for an integer $s$, let
\begin{equation*}
    O_s =\begin{bmatrix} C^\top  & (CA)^\top & \ldots & \Paren{ CA^{s-1} }^\top  \end{bmatrix}^{\top}.
\end{equation*}
A linear dynamical system is \emph{observable} if for some $s$, the matrix $O_s$ has full column rank. Intuitively, this condition ensures that there is no portion of the state space that we cannot observe eventually. 

%Similarly, we need to ensure that the control input is not degenerate, i.e. we want to rule out the possibility that it only acts in a subspace that is not spanned by rows of $A$. If that were to happen, for example if the matrix $B$ projects the control input into the kernel of $A$, the observations would not contain any information about the input. This is made precise by 
Now consider the \emph{controllability matrix}: for an integer $s$, let
\begin{equation*}
    Q_s = \begin{bmatrix} B & AB & \ldots & A^{s-1} B\end{bmatrix}
\end{equation*}
A linear dynamical system is \emph{controllable} if the controllability matrix has full row rank. Intuitively this condition ensures that there is no portion of the state space that cannot be reached by the appropriate inputs. If either observability or controllability are violated, it is information-theoretically impossible to learn. 

While the full rank conditions are enough to build an asymptotic theory, we will need natural quantitative counterparts to get finite sample guarantees. In particular we assume that $O_s$ and $Q_s$ have bounded condition number for some $s$. These assumptions are usually made in addition to the ones from the previous subsection, as they are needed in the stability bounds for the Ho-Kalman algorithm \cite{oymak2019non}. Furthermore we show that (see Theorem~\ref{thm:informal-lower-bound}) they are information-theoretically necessary in order to learn a linear dynamical system from a polynomial length trajectory. 

Finally, as is standard, we also assume that the system is \emph{non-explosive}, i.e. the eigenvalues of $A$ are bounded by $1$ in magnitude. Note that assuming the eigenvalues are bounded is much weaker than assuming the singular values are bounded (e.g. consider the types of upper triangular matrices that arise in control theory, including $n$-dimensional integrators~\cite{rowell2002state}).

\paragraph{Relaxed Control and Noise.} 

In the literature, the standard assumption is that the initial state, the process and observation noise are all drawn from a Gaussian. But Gaussianity is not meant to literally be true and it is often assumed for convenience. We show that we can dramatically relax this assumption to allow heavy-tailed distributions instead. In particular, for the control input, $u_t$, we only require the underlying distribution to have well-behaved fourth-moments:

\begin{definition}[(4,2)-Hypercontactivity]
A distribution $\calD$ over $\mathbb{R}^d$ is $(4,2)$-hypercontractive if for all $v$, 
\begin{equation*}
    \expecf{x\sim \calD}{ \Iprod{x, v}^4 } \leq O(1) \expecf{x\sim \calD}{\Iprod{x,v}^2}^2.
\end{equation*}
\end{definition}

We note that several families of distributions are hypercontractive, including Gaussians, uniform distributions over the hypercube, sphere and other convex bodies, the Laplace, gamma, chi-squared, Wishart, Dirichlet and beta distributions, and in general, all log-concave distributions. Further, the set of hypercontractive distributions is closed under affine transformations, products and mixtures. 

Finally, we only require that the distributions of the process noise, $w_t$, and observation noise, $z_t$, have bounded covariance. We state these assumptions formally in Section \ref{sec:formal-setup}. 

%Let $A \in \mathbb{R}^{n \times n}$, $B \in \mathbb{R}^{n \times p}$,  $C \in \mathbb{R}^{m \times n}$, and $D\in \mathbb{R}^{m \times p}$. Then, given $\{ u_i, y_i \}^{\bar{N} }_{i=1}$, such that $u_i \sim \calN(0, I)$ (note that this is without loss of generality because we can absorb a linear transformation into $B$ and $D$) consider the following dynamics: 

%\begin{equation}
%\begin{split}
%    x_{t} &= A x_{t-1} + B u_t + w_t ,\\
%    y_{t} &= C x_t + D u_t + z_t ,
%\end{split}
%\end{equation}
%where $w_t \sim \calN(0, \Sigma_w)$ and $z_t \sim \calN(0, \Sigma_z)$.  

\subsection{Our Results}

Our approach is based on the \emph{method-of-moments} rather than least-squares regression.  Our starting point is the following folklore observation: for any integers $t>j$, 
\begin{equation}
\label{eqn:unbiased-est}
    \expecf{u_t \sim \calD_u }{ y_{t+j} u_t^\top } = \begin{cases} D \hspace{0.6in} \text{ if } j = 0 \\
        C A^{ j - 1} B \hspace{0.2in}\text{ otherwise}
        \end{cases}
\end{equation}
%This (almost folklore) observation simply follows from recalling that the control inputs are independent, i.e. for any $t \neq t'$, $\expecf{}{u_t u_{t'}} = 0$. 
However getting accurate estimates of the Markov parameters is a challenging task. For a fixed $j$, since the expectation of the estimator $y_{t + j} u_t^\top$ does not depend on $t$, a natural approach to estimate 
$\expecf{}{y_{t+j} u_t^\top}$ is to average over several control-observation pairs: $\widehat{C A^{j-1} B} =\frac{1}{T} \sum_{t \in [T]} y_{t+j}u_t^\top$ and hope that this estimator converges to its expectation. Unfortunately, this is just not true! 

The first issue is that samples of the form $y_{t+j} u_{t}^\top$ are not independent for different values of $t$. 
The second issue is that, in the marginally stable setting, the variance of this statistic grows with $t$, even when the control and the noise are Gaussian (see Lemma~\ref{lem:variance-blowup} for a simple example).  Thus, directly using the empirical estimate can be highly inaccurate \textit{no matter how long our trajectory is}. One of the key steps in our algorithm is to learn a transformation of the observations to a new time series $\Set{\hat{y}_1, \hat{y}_2, \ldots , \hat{y}_T}$ such that $\expecf{}{\hat{y}_{t+j} u^\top_t} = CA^{j-1}B$ and the variance of our estimator is bounded.  As a result, we obtain the following theorem:

\begin{theorem}[Efficiently Learning a Linear Dynamical System, informal Theorem~\ref{thm:main-learning-lds}]
Given $\epsilon>0$, a fixed polynomial length trajectory from a linear dynamical system satisfying mild non-degeneracy assumptions (see Subsection \ref{subec:our-assumptions-informal}), there exists an algorithm that outputs estimates $\hat{A}, \hat{B}, \hat{C} , \hat{D}$ such that with probability at least $9/10$, there exists a similarity transform $U$ satisfying 
\begin{equation*}
\begin{split}
    \Norm{A - U^{-1} \hat{A}  U }  \leq
    \epsilon  , 
    \Norm{B - U^{-1} \hat{B} }  \leq   \epsilon, 
    \Norm{C - \hat{C}U }  \leq  \epsilon, 
    \Norm{D - \hat{D}}  \leq   \epsilon .
\end{split}
\end{equation*}
Further, the algorithm runs in time that is  a fixed polynomial in all the parameters. 
\end{theorem}
\begin{remark}
Note that it is only possible to recover the system parameters up to some global transformation $U$ since all such transformations lead to equivalent dynamics, see e.g. \cite{oymak2019non}.
\end{remark}

%We emphasize that our guarantees are novel in several respects. First, we give the first polynomial time algorithm for the notoriously challenging marginally stable case. Second, our algorithm works under much weaker distributional assumptions. It is standard to assume that the initial state, the process and observation noise are all Gaussian. But Gaussianity is not meant to literally be true. It is only supposed to be for convenience. Our algorithm works just assuming hypercontractivity. Thus it fits into the broader body of work in theoretical computer science of revisiting high-dimensional estimation problems under more flexible distributional assumptions. Third, we show that our assumptions are essentially necessary:

The main appeal of our algorithm is that it works in essentially the most general setting possible. In particular we show the following lower bound:

\begin{theorem}[Sample Complexity Lower Bound for Ill-Conditioned Systems]
\label{thm:informal-lower-bound}[Informal, see Theorem~\ref{thm:lower-bound}]
If for an LDS, the observability matrix $O_s$  has smallest singular value less than $\delta$ for all orders $s$, then any algorithm that uses less than $\sim 1/\sqrt{\delta}$ length trajectories incurs constant error in estimating $A,B,C,D$ with constant probability.  The same statement holds with the observability matrix $O_s$ replaced by the controllability matrix $Q_s$.  In particular, if $\delta$ is exponentially small, then an exponential number of samples are required to learn the parameters.
\end{theorem}

It turns out that super-resolution \cite{donoho1992superresolution, candes2014towards}, namely the task of recovering a sparse signal from noisy low-frequency measurements, corresponds to a special case of learning linear dynamical systems. It is known that super-resolution exhibits a sharp phase transition, where the problem goes from having efficient algorithms with polynomial running time and sample complexity, to being information-theoretically impossible, unless the noise is exponentially small \cite{moitra2015super}. Thus there are some linear dynamical systems where it is impossible to learn the true parameters with bounded length trajectories. We refine this connection to show instance-wise lower bounds for learning any linear dynamical system whose observability or controllability matrices are close to singular. Thus the assumptions our algorithm needs are qualitatively tight, and our results close the question of what linear dynamical systems can be efficiently learned. 

%Finally for strictly stable linear dynamical systems we show that naive unstabilized method of moments works too. While previous analyses for linear regression get quite involved, returning to the method of moments make everything much easier.  

%We have learned all LDS's. 
\section{Related Work}
\label{sec:related-work}
\paragraph{Linear Time Invariant Systems: Identification, Prediction, Estimation.}
There is a long history of identifying linear dynamical systems from measurements, see \cite{Galrinho2016LeastSM} for extensive references.  A focus of these works is on the "pre-filtering" approach to handling long range correlations in learning dynamical systems, see \cite{ding2013} \cite{zhang2011} \cite{spinelli2005} \cite{simchowitz2019learning}.  Recently, there is a flurry of work on prediction and estimation for LDS's through the framework of no regret learning both in the fully observable setting \cite{simchowitz2018} \cite{sarkar2019near} \cite{faradonbeh2017} and the partially observed setting \cite{shah2012linear} \cite{hardt2018gradient}.  For a variety of assumptions on the dynamics matrix, such as diagonalizability, there is work on learning marginally stable LDS's \cite{hazan2018spectral} \cite{hazan2017}.  Many works take a regression approach to estimating the markov parameters of the LDS for strictly stable systems see \cite{lee2020} \cite{SRD2021} \cite{salar2020} \cite{DM2022}.  In these settings it is possible to take advantage of the decay of the coefficients of the associated regressors.  Marginal stability can be handled with multiple trajectories see \cite{zl2021} \cite{SOF2022}.  Closed loop system identification has also been studied see \cite{LL2019} \cite{LAHA2020}.  

Somewhat related to our work is the problem of prediction without system identification in marginally stable LDS's \cite{tsiamis2019finite} \cite{ghai2020} but with an assumption on the exponential decay of the kalman filter coefficients.  In  \cite{RJR2020}, the exponential stability of the Kalman filter assumption is removed via a procedure that builds a succinct bank of filters for the prediction task and with an additional assumption on the dynamics having real eigenvalues.  For a survey of the area see \cite{TZMP2022}.

\paragraph{Relaxing distributional assumptions.}
In recent years, there has been a tremendous amount of work on designing algorithms that do not rely on strong distributional assumptions, such as Gaussianity, and only require much milder conditions. In particular, hypercontractivity of linear functions and low-degree polynomials have been identified as key analytic conditions that admit efficient algorithms for numerous problems in high-dimensional algorithmic statistics. In particular, souped up variants of \emph{hypercontractivity} are used for
heavy-tailed mean and covariance estimation~\cite{lugosi2019sub, hopkins2018sub,cherapanamjeri2020algorithms}, robust moment estimation~\cite{kothari2017outlier}, robust regression~\cite{klivans2018efficient, prasad2018robust, bakshi2021robust, zhu2020robust,cherapanamjeri2020optimal, pensia2020robust, jambulapati2021robust}, robustly clustering mixture models~\cite{hopkins2018mixture,kothari2018robust,bakshi2020outlier,diakonikolas2020robustly} and list-decodable learning~\cite{karmalkar2019list, raghavendra2020regression, raghavendra2020list, bakshi2020list, cherapanamjeri2020list, ivkov2022list}. Algorithms with relaxed distributional assumptions were also recently given for online regression \cite{chen2022online} and Kalman filtering \cite{chen2022kalman}.

\section{Technical Overview}\label{sec:technical-overview}

In this section, we describe our key algorithmic ideas and the corresponding technical challenges involved. 

\subsection{A Thought Experiment} Consider the setting where we already know the parameters, $A,B,C$ and $D$, of the underlying linear dynamical system. While there is nothing left to learn in this setting, we can still ask whether there exists a transformation of the observations $\Set{ y_t }_{t \in [0, T]}$ to a new time series $\Set{ \hat{y}_t }_{t \in [0, T]}$, such that the variance of the random variable $\hat{y}_{t + j} u_t$ is bounded.

It is indeed possible to do so by considering a simple linear transformation of the observations: let $\hat{y}_t = y_t - \sum_{j =1}^{n} c_j y_{t-j}$, where the $c_j$'s are the coefficients of the characteristic polynomial of $A$. To see why this works, we recall that by the Cayley-Hamilton theorem (see Fact \ref{fact:cayley-hamilton}), the coefficients of the characteristic polynomial satisfy the following algebraic identity: 

\begin{equation}
\label{eqn:char-identity}
    A^{n} - \sum_{j \in [n]} c_j A^{n-j} = 0
\end{equation}

Therefore, assuming (for the purposes of exposition) that $w_t$'s and $z_t$'s are bounded, and $D=0$, we have
\begin{equation}
\label{eqn:hatyt-split}
\begin{split}
    \hat{y}_t  & = y_t - \sum_{j =1}^{n} c_j y_{t-j}\\
    % & = \sum_{i = 0}^{t}\Paren{ CA^{i-1} B u_{t-i} }  - \sum_{j =1}^{n} c_j \Paren{ \sum_{i = 0}^{t-j}\Paren{ CA^{i-1} B u_{t-j-i} }  } \\
    & = \underbrace{ \sum_{i = 1}^{n} \Paren{ C A^{i-1} B - \sum_{j = 1}^{i - 1} c_j CA^{i-j - 1}  B  } u_{t-i} }_{(a)}  +  \underbrace{ \sum_{i = n+1}^{t} \Paren{ C A^{ i-n-1} \Paren{ A^{n} - \sum_{j = 1}^{n} c_j  A^{n-j} }  B  } u_{t-i}}_{(b)}
\end{split}
\end{equation}
We note that  term (a)  above only has $n$ terms and does not grow as a function of $t$, and term (b) is in fact zero, since we can repeatedly apply the identity from Equation \eqref{eqn:char-identity}. A simple computation then implies that the estimator $\hat{y}_{t+j} u_{t}^\top$ satisfies $\expecf{}{ \hat{y}_{t+j} u_{t}^\top } = C A^{j-1} B$ and has bounded variance. We replicate this thought experiment, by learning the coefficients that stabilize the variance of our estimator from the observations directly. We dedicate the rest of the technical overview to describe how we accomplish this task.

\subsection{Learning the Stabilizing Transform}

For ease of exposition, we assume that $C$ is a $1 \times n$ matrix and therefore the resulting observations, $y_t$, are scalars. A natural approach is to then consider the following least-squares regression problem:

\begin{equation*}
    \min_{c_1, c_2, \ldots c_n } \sum_{t \in [T]} \Paren{ y_t - \sum_{j\in [n]} c_j y_{t-j} }^2
\end{equation*}

We know that the coefficients of the characteristic polynomial are a feasible solution to this regression problem, and the resulting linear transformation of the observations results in an estimator with bounded variance.   Such an approach also appears in~\cite{simchowitz2019learning} but they incur unspecified, potentially exponential dependencies on the system parameters due to the complexities of analyzing this regression problem directly.    In particular, we do not have fine-grained control over the solution returned by solving the regression problem, and apriori, the regression solution need not be close to the coefficients of the characteristic polynomial of $A$.

\paragraph{Convex Program.} Instead, we take a more direct approach to stabilizing the variance and consider a different convex program, specifically designed to do so. In particular, we find a vector $\alpha = \Paren{\alpha_1, \alpha_2, \ldots \alpha_s}$ such that the following constraint system is feasible:
\begin{equation}
    \calC_{\alpha}=
  \left \{
    \begin{aligned}
      & \forall j \in [s]
      & \abs{\alpha_j}^2 
      & \leq  P_0\\
      & \forall i \in [T]
      &  \Big| y_{i + k }  - \sum_{j \in [s]}  \alpha_j\cdot  y_{i -j } \Big|^2 
      & \leq  P_1  \\
    \end{aligned}
  \right \},
\end{equation}
where $s$ is the integer satisfying the observability and controlability assumptions from Definition~\ref{def:system-assumptions}, and $P_0$ and $P_1$ are sufficiently large polynomials in the system parameters (see Algorithm \ref{algo:stabilizing} for details).  

Intuitively, the first constraint posits that each coefficient, $\alpha_j$ is bounded in magnitude. This is necessary since the process and observation noise scale proportional to the coefficients in the linear transformation, and we cannot afford to pay exponentially in these quantities. The second constraint posits that the resulting observations themselves are bounded, and tries to enforce a universal bound on the variance of each $\hat{y}_t$ appearing in the estimator $\frac{1}{T}\sum_{t\in [T]} \hat{y}_{t+k}u_t^\top$.  
% \abnote{anyone have any better intuition for what the second constraint is trying to do? }

\paragraph{Feasibility.} Observe, in contrast to the characteristic polynomial, we are only taking a linear combination of the previous $s$ (potentially $\ll n$) observations. 
% \footnote{\textcolor{red}{Allen: this point is kind of moot if we assume C is 1xn.  You basically have to go out to n in this case.  I added another sentence to clarify this}} 
While this difference does not manifest itself when $C$ is $1 \times n$, it becomes crucial when $C$ is $m \times n$ for $m > 1$ for obtaining guarantees that depend only on the observability and controllability matrix.  Also, note that we are expressing $y_{t+k}$, rather than $y_t$, as a linear combination of $y_{t-1}, \dots y_{y-s}$.  This difference is also crucial in the construction and analysis of our estimator.

In order to establish feasibility of $\calC_{\alpha}$, we invoke the observability assumption: since $O_s$ has bounded condition number, there exists a vector $\alpha^* = \Paren{\alpha_1^* , \ldots , \alpha_s^*}$ such that each $\alpha$ is bounded and the following identity holds:

\begin{equation*}
    CA^{k+s} - \alpha_1 CA^{s-1} - \alpha_2 CA^{s-2} - \ldots - \alpha_s C  = 0.
\end{equation*}
We then follow an argument similar to the one in Equation \eqref{eqn:hatyt-split} to show that the magnitude of $\hat{y}_t$'s is bounded (see Lemma \ref{lem:feasibility} for details). We also note that the above program is convex, and admits an efficient separation oracle, and therefore, we can find a feasible $\alpha$ in polynomial time. Interestingly, the feasibility analysis only requires that the covariance of the control input, process noise and observation noise be bounded, and does not require strong assumptions such as sub-Gaussian tails. 

\paragraph{The Anti-concentration Potential.} Next, we show that any feasible solution to the constraint system actually yields a stabilized estimator. To accomplish this goal, we design a potential function that captures the variance of our estimator, and argue that if the potential is large, with high probability, some constraint in $\calC_{\alpha}$ must be violated. In particular, for any vector $\alpha$, and integer $l$, we consider the potential 
\begin{equation*}
    \calG_{\alpha, l} = \sum_{i=0}^{l} \norm{ F_{\alpha}(A) A^i B }_F^2
\end{equation*}
where 
\[
F_{\alpha}(A) = CA^{k+s} - \alpha_1 CA^{s-1} - \alpha_2 CA^{s-2} - \ldots - \alpha_s C \,.
\]
We observe that the terms appearing in this potential are the trace of  $\expecf{}{ \Paren{  F_{\alpha} A^{i} B u_{t-i} }^\top \Paren{  F_{\alpha} A^{i} B u_{t-i} } }$, which captures how large intermediate terms are, as a function of $\alpha$. In particular, if $\alpha= \alpha^*$, the trace would be $0$. We make this intuition precise in Lemma~\ref{claim:potential-is-variance}.

% \abnote{define $F_{\alpha}$.}

Next, we show that we can split up the terms appearing in $\hat{y}_{t+k}$ into three parts as follows:
\begin{equation*}
    \hat{y}_{t+k} = X_t + V_t + W_t, 
\end{equation*}
where $X_t$ is polynomially bounded in the system parameters, $V_t$ is a random variable such that the covariance matrix of $V_t$, denoted by $\Sigma_{V_t}$ satisfies $\textsf{Tr}(\Sigma_{V_t}) = \mathcal{G}_{\alpha, l}$, and $W_t$ is a random variable that we do not have control over, and may potentially be unbounded. This presents obstacle since $W_t$ can wipe out the information contained in $V_t$, and $\calG_{\alpha,l}$ may be large without violating any constraint in $\calC_{\alpha}$. 

Here, we observe that such an event can be avoided precisely when the random variable $V_t$ is \emph{anti-concentrated}, i.e. the probability that $V_t$ lands in a ball of small radius is small. Perhaps counter-intuitively, we show that if the $4$-th moment of $V_t$ concentrates, then it already possesses the \emph{anti-concentration} properties we require. We make this precise in Lemma~\ref{claim:anti-concentration}, where we establish a Payley-Zigmund style inequality, showing that if a random variable is  $(4,2)$-hypercontractive (see Definition \ref{def:hypercontractivity}), then the probability it lands in any interval that is a constant fraction of it's variance is bounded by a constant. Finally, we show that $V_t$ is $(4,2)$-hypercontractive if the control inputs are $(4,2)$-hypercontractive, and therefore, significantly relax the Gaussianity assumption. 

To summarize, we show that if the potential is large, the magnitude of $V_t$ is large, and since $V_t$ is anti-concentrated, $W_t$ cannot \textit{wash away} this information. Therefore,  $\abs{\hat{y}_{t+k}}$ must be large, for some $t$, which is a contradiction to the feasibility of $\calC_{\alpha}$.

\paragraph{Dependent Random Variables and Decoupling.} We then establish that if the potential $\calG_{\alpha}$ is small, for a fixed setting of $\alpha$, the resulting estimator has bounded variance: 
\begin{equation}
\label{eqn:var-bound}
    \expecf{}{ \Norm{ \frac{1}{L} \sum_{t \in [L]} \hat{y}_{t+j} u_{t}^\top  - C A^{j-1}B }^2} \leq \frac{P_1}{L} \Paren{ \norm{\alpha}^2 + G_{\alpha,L} },
\end{equation}
for some fixed polynomial $P_1$ in the system parameters.  We treat $L$ as a sufficiently large polynomial in the system parameters and $1/\eps$ (where $\eps$ is the desired accuracy).

This argument is fairly involved and heavily uses the independence of the $u_t, w_t$ and $z_t$'s.  We refer the reader to Lemma~\ref{lem:variance-bound-v2} for a complete proof. While the above inequality holds for a fixed setting of $\alpha$, we note that the $\alpha$'s output by solving the constraint $\calC_{\alpha}$ themselves depend on the randomness in the control input and the noise non-trivially. 

To overcome this issue, we decouple the $\alpha$'s from $u_t$'s and establish a symbolic matrix inequality, where the matrices only depend on the $u_t$'s. Here, we treat the vector $v_{\alpha} = (1,\alpha_1, \ldots, \alpha_s)$ as a formal variable, and write the potential as a quadratic form in the vector $v_\alpha$:
\begin{equation*}
    \calG_{\alpha, L} =  v_{\alpha}^\top G_{L} v_{\alpha},
\end{equation*}
where $G_L$ is a PSD matrix. We note that such a representation always exists and is unique since $\calG_{\alpha, L}$ is a sum-of-squares in $\alpha$. Similarly, we observe that the variance we want to bound admits a similar decomposition: let $M_j$ be the PSD matrix such that
\begin{equation*}
    \expecf{}{ \Norm{ \frac{1}{L} \sum_{t \in [L]} \hat{y}_{t+j} u_{t}^\top - C A^{j-1}B  }^2} = v^\top_{\alpha} M_j v_{\alpha}.  
\end{equation*}

Observe, the matrices $G_{L}$ and $M_j$ are independent of the formal variables $\alpha$, and in Corollary~\ref{coro:estimator-variance-symbolic} we establish the following inequality:
\begin{equation}
\label{eqn:mat-inequality-intro}
    \expecf{u,w,z}{ M_j } \preceq \frac{P_1}{L}\Paren{I + G_L}
\end{equation}
Since the above inequality holds for all quadratic forms simultaneously, one natural way to proceed would be to consider an $\epsilon$-net over the $\alpha$'s and union bound over each vector in the net satisfying 
%that $G_{\alpha , L}$ is not too large.
the quadratic form in Equation \eqref{eqn:mat-inequality-intro}. 
To execute this, we need a net that is fine enough to account for how much error we accumulate in a term of the from $\hat{y}_{t+j} u_t^\top$. Note, the largest terms we need to account for are roughly of the form $\alpha_j A^{L}$, which naively requires a $1/\norm{A}^L$-net. Unfortunately, this net is too fine and we cannot afford to union bound over all the vectors in this net  because in some sense we only have $L$ samples. 

\paragraph{Bounded Eigenvalues to Smaller Nets.} To address the issue above, we show that for any $n\times n$ matrix $A$ with complex entries, if the eigenvalues of $A$ are bounded by $1$ in magnitude, the operator norm of $A^L$ in fact grows as $L^n$, instead of exponentially in $L$ (see Lemma \ref{claim:powering-bound} for a precise statement). Here, we crucially note that we only assume the eigenvalues, instead of the singular values (which would make this statement trivial but would rule out several important families of linear dynamical systems), are bounded. With this insight, we can then union bound over all vectors $\alpha$ in a $1/L^n$-net, and establish equation ~\eqref{eqn:var-bound} for the $\alpha$ output by solving the constraint system $\calC_{\alpha}$. To conclude, we have shown that the variance of our estimator is bounded, and therefore, we can estimate each block of the Markov Parameter matrix.

\subsection{Lower Bound for Ill-Conditioned LDS's}
In light of our results on learning LDS's, a natural question is whether the assumptions on observability and controllability are necessary.  We exhibit information-theoretic lower bounds on the sample complexity of learning a linear dynamical systems with Gaussian noise and Gaussian inputs, when the observability or controllability matrices are exponentially ill conditioned.  For this section, we discuss the case when the observability matrix is ill conditioned.  The case when the controllabiltiy matrix is ill conditioned follows essentially the same argument.
\begin{definition}
We say that an LDS $\calL(A,B,C,D)$ is $(\delta,v)$-unobservable if $v$ is a unit vector such that for all integers $s \geq 0$,
\[
\norm{CA^s v} \leq \delta \,.
\]
\end{definition}

Now the key to proving an information-theoretic lower bound is the observation that when the input and noise distributions are Gaussian, the system measurements and control inputs $(u_0,...,u_T,y_0,...,y_T)$ are a Gaussian process.  Therefore, the joint distribution is uniquely determined by its covariance matrix.  On the other hand, we can explicitly compute the covariance matrix in terms of the system parameters $A,B,C,D$.  While there are several terms in the expression (see Fact~\ref{fact:obs-cov-formula}), the main point is that essentially all of the terms look like $CA^jB$.  Now, when the system is $(\delta, v)$-unobservable, we can replace $B$ with $B + vu^\top$ for an arbitrary vector $u \in \R^p$ while only changing expressions of the form $CA^jB$ by a little bit.  

Overall, we can show that the pair of LDS's $\calL = \calL(A,B,C,D)$ and $\calL' = \calL(A,B + vu^\top,C,D)$ are statistically close up to time $T \ll 1/\delta$. This means that no algorithm using $\ll 1/\delta$ length trajectories can distinguish the two systems and since their parameters are not close to equivalent up to similarity (for generic choices of parameters), the algorithm must incur large error.

\section{Formal Setup}
\label{sec:formal-setup}
In this section, we formally state the linear dynamical system model, and our assumptions.  

\begin{model}[Linear Dynamical System]
\label{model:LDS}
 Let $A \in \mathbb{C}^{n \times n}$, $B \in \mathbb{C}^{n \times p}$,  $C \in \mathbb{C}^{m \times n}$, and $D\in \mathbb{C}^{m \times p}$ be complex valued matrices. Let $\calD_0$, $\calD_{u}$, $\calD_{w}$, $\calD_{z}$ be distributions with mean zero. Then, a Linear Dynamical System, $\calL\Paren{A,B,C,D}$, is defined as follows: 
\begin{equation*}
\begin{split}
    x_{t+1} &= A x_{t} + B u_t + w_t ,\\
    y_{t} &= C x_t + D u_t + z_t ,
\end{split}
\end{equation*}
where $x_0 \sim \calD_{0}$, and for all $t\in\mathbb{N}$,  $u_t \sim \calD_{u} $, $w_t \sim \calD_{w}$ and $z_t \sim \calD_{z}$. 
\end{model}

We see only the sequence of observations $y_1, y_2, \dots , y_T$ up to some time $T$ and our goal is to learn the parameters of the system $A,B,C,D$.  We need some assumptions about the parameters and also the input and noise distributions which we discuss below (as otherwise the system may be degenerate and it may be information-theoretically impossible to learn, see Section~\ref{sec:lower-bound}).

\subsection{Assumptions on the system parameters}
% LDS are typically studied under 
% Gaussian input, but real world does not behave like Gaussian. 

We begin by ensuring that the 
linear dynamical system at hand is not degenerate.  This notion can be made precise by considering the \emph{observability matrix}: 
\begin{definition}[Observability Matrix]\label{def:observability}
For an integer $s$, define the matrix $O_s \in \R^{sm \times n}$ as
\begin{equation*}
    O_s =\begin{bmatrix} C \\ CA \\
    \vdots \\
    CA^{s-1} \end{bmatrix}.
\end{equation*}
\end{definition}

A LDS is \emph{observable} if for some $s$, the matrix $O_s$ has full column-rank. 
Similarly, we need to ensure that the control input is not degenerate, and only acts in a subspace that is not spanned by $A$. This is made precise by considering the \emph{controllability matrix}: 
\begin{definition}[Controllability Matrix]\label{def:controllability}
For an integer $s$, define the matrix $Q_s \in \R^{sp \times n}$ as
\begin{equation*}
    Q_s = \begin{bmatrix} B & AB & \ldots & A^{s-1} B\end{bmatrix}
\end{equation*}
\end{definition}
A LDS is \emph{controllable} is the controllability matrix has full row-rank. We note that we assume a quantitative strengthening of these two assumptions to $O_s$ and $Q_s$ having bounded condition number (and this is necessary, recall Theorem \ref{thm:informal-lower-bound}). 

\begin{definition}[Well-Behaved Linear Dynamical System]\label{def:system-assumptions}
We say a linear dynamical system $\calL(A,B,C,D)$ is \emph{well-behaved} if the following assumptions hold:
\begin{enumerate}
    \item \textbf{Non-trivial Controller.} The matrix $B$ satisfies $\norm{B}\geq 1$.
    \item \textbf{Non-trivial Measurement.} The matrix $C$ satisfies $\norm{C}\geq 1$.
    \item \textbf{Non-exposive System.} All eigenvalues of $A$ have magnitude at most $1$. 
    \item \textbf{Bounded Condition Number.} $O_s$ has full column-rank, $Q_s$ has full row-rank and for some integer $s$, and parameter $\kappa\geq1$,
    \begin{equation*}
        \begin{split}
            \sigma_{\max}(O_{2s})/\sigma_{\min}(O_s) &\leq \kappa, \\
            \sigma_{\max}(Q_{2s})/\sigma_{\min}(Q_s) &\leq \kappa.
        \end{split}
    \end{equation*}
\end{enumerate}
\end{definition}

\begin{remark}
Crucially, the above assumption is on the eigenvalues and not the singular values of $A$, which would be a far stronger assumption, as discussed in the introduction. 
\end{remark}

\begin{remark}
While the bounded condition number assumption is standard in the literature~\cite{fetzer1975observability, muller1972analysis}, in Section \ref{sec:lower-bound}, we show that a polynomial bound on the condition number is necessary in the sample complexity, even information-theoretically. As a consequence, it is impossible to learn an exponentially ill-conditioned system with polynomially bounded observations. Finally,  note that up to polynomial factors, it suffices to have a bound on $\sigma_{\max}(O_{t})/\sigma_{\min}(O_s)$ as long as $t/s > 1 + c$ for some positive constant $c$ (see Claim~\ref{claim:boundA}).  For simplicity, we wrote the above condition for $t = 2s$.
\end{remark}

\subsection{Assumptions on the distribution of the control and noise}

We consider the following assumptions over the control input, system and process noise distributions:

\begin{definition}[Distributional Assumptions]\label{def:distributional-assumptions}
For all $t \in [T]$, we assume that  $u_t \sim \calD_{u}$,  $w_t\sim \calD_{w}$ and $z_t \sim \calD_{z}$ are each sampled independently from the corresponding distributions. Additionally, $x_0 \sim \calD_0$. Then, 
\begin{itemize}
    \item \textbf{Mean Zero:} $\calD_{u}$, $\calD_{w}$,  $\calD_{z}$ and $\calD_{0}$ are all mean $0$ distribution. 
    \item \textbf{Isotropic and Hypercontractive Control:} The covariance of $\calD_u$,  $\Sigma_{\calD_u } = I$, and $\calD_u$ is $(4,2,K)$-hypercontractive for a fixed constant $K \geq 3$ (see Definition~\ref{def:hypercontractivity}).
    \item \textbf{Bounded Variance Noise:} For $\sigma_w, \sigma_z \geq 1$, the covariances of $\calD_w$ and $\calD_z$ satisfy  $ \Sigma_{\calD_w} \preceq \sigma_w I$ and $\Sigma_{\calD_z} \preceq \sigma_z I$.

    \item \textbf{Starting Point:}  The distribution $\calD_0$ has covariance $\Sigma_{\calD_0} \preceq \sigma_0 I$

\end{itemize}
\end{definition}

\begin{remark}\label{rem:not-iid}
We don't actually need that the $u_i, w_i, z_i$ are all drawn from the same distribution across different time-steps.  We only need that they are independent.  In other words, all of our results still hold if we allow for there to be different distributions $\calD_{u,t},\calD_{w,t}, \calD_{z,t}$ at each time-step that all satisfy the above assumptions. 
\end{remark}

\section{Preliminaries}

%Given a matrix $A \in \mathbb{C}^{n \times n}$ we can compute its eigenvalue decomposition, denoted by $ Q\Lambda Q^{-1}$, such that $Q$ is an $n \times n$ matrix with linearly independent columns as eigenvectors and $\Lambda$ is a diagonal matrix with the eigenvalues. Further, 
We begin with some notation and basic facts from linear algebra and probability. For a matrix $A \in \C^{n \times n}$, we use $\norm{A} = \norm{A}_{\textrm{op} } = \max_{\norm{u}=1 } \norm{Au}_2$ and $\norm{A}_F = \sqrt{ \sum_{ i,j \in [n]} \abs{A_{i,j}}^2  }$. We use the notation  $A^{\top}$ to denote the transpose when $A$ only has real entries. Further, for $A \in \C^{n \times m}$ such that $n \geq m$,  let $\textsf{SVD}(A)= U \Sigma V^\top$ denote the singular value decomposition of $A$, where $U \in \mathbb{C}^{n \times m}$ and $V^\top \in \mathbb{C}^{m \times m}$ are unitary matrices (see Definition~\ref{def:unitary-matrices}, and $\Sigma$ is a diagonal matrix, with the singular values denoted by $\sigma_1\geq \sigma_2 \geq \ldots \sigma_m\geq 0$. 

\subsection{Linear Algebra Background}

\begin{definition}[Unitary Matrices]
\label{def:unitary-matrices}
Given a symmetric matrix $U \in \mathbb{C}^{n\times n}$ we say  $U$ is a unitary matrix  if $U^\top U = U U^\top = I$.
\end{definition}

\begin{fact}[Operator Norm of Unitary Matrices]
\label{fact:submult-op}
If $Q \in \mathbb{C}^{n\times n}$ is a unitary matrix, $\norm{Q}_{\textrm{op}} = 1$.
\end{fact}

\begin{fact}[Sub-Multiplicativity of Operator Norms]
Given matrices $A \in \mathbb{C}^{n \times d}, B \in \mathbb{C}^{d \times m}$, $\norm{AB}_{\textrm{op}} \leq \norm{A}_{\textrm{op}}\cdot \norm{B}_{\textrm{op}}$.
\end{fact}

\begin{fact}[Cayley-Hamilton Theorem]
\label{fact:cayley-hamilton}
Given a square matrix $A \in \mathbb{C}^{n\times n}$, the characteristic polynomial of $A$ is defined as $p_{A}\Paren{\lambda} = \textsf{det}\Paren{\lambda I - A} = \lambda^n + c_{n-1} \lambda^{n-1}+ \ldots c_1 \lambda + c_0$, where the coefficients, $c_i$, are scalar. Then, consider the matrix valued polynomial $p_{A}\Paren{A}= A^n + c_{n-1} A^{n-1}+ \ldots c_1 A + c_0 I$. The Cayley-Hamilton theorem states $p_{A}\Paren{A}=0$. 
\end{fact}

\begin{fact}\label{claim:upper-triangular}
For any matrix $A \in \C^{n \times n}$, there is a unitary matrix $Q$ such that $Q^{-1}AQ$ is upper triangular.
\end{fact}
\begin{proof}
$A$ must have some eigenvector, say $v$.  Also normalize $v$ so that it is a unit vector.  Let $Q_0$ be a matrix whose first column is $v$ and whose columns form an orthonormal basis of $\C^n$.  Then $A' = Q_0^{-1}AQ_0$ has all entries in the first column equal to $0$ except possibly the first entry.  Now it suffices to compute a unitary matrix in $\C^{(n-1) \times (n-1)}$ that transforms the $(n-1) \times (n-1)$ submatrix of $A'$ (excluding the first row and column) into an upper triangular matrix but this can be done by induction.
\end{proof}

We also establish the following key lemma to upper bound the operator norm of matrix polynomials, when the underlying matrix has bounded eigenvalues. 

\begin{lemma}[Opertor Norm of a matrix with bounded Eigenvalues ]\label{claim:powering-bound}
Let $A \in \C^{n \times n}$ be a matrix and assume that all eigenvalues of $A$ have magnitude at most $1$.  Then for any integer $L$,
\[
\norm{A^L} \leq n \cdot (2(1 + \norm{A}) L)^n \,.
\]
\end{lemma}
\begin{proof}
By Fact~\ref{claim:upper-triangular} and Fact~ \ref{fact:submult-op}, without loss of generality we can assume that $A$ is upper triangular.  Then, all of its diagonal entries are eigenvalues so all of its diagonal entries have magnitude at most $1$,  Now we bound the magnitude of all entries of $A^L$.  Consider the entry indexed by $i,j$.  Clearly $i \leq j$ or the corresponding entry is $0$.  Next, by definition of $(A^L)_{i,j}$, 
\begin{equation}
\begin{split}
|(A^L)_{ij}| & = \sum_{\substack{i_1, \dots , i_L \\ i \leq i_1 \leq \dots \leq i_{L-1} \leq j} } A_{ii_1}\cdots A_{i_{L-1}j} \\
& \leq \sum_{\substack{i_1, \dots , i_L \\ i \leq i_1 \leq \dots \leq i_{L-1} \leq j} } |A_{ii_1}| \cdots |A_{i_{L-1}j}| \\
& \leq \binom{L - 1 + n}{n}(1 + \norm{A})^n \\
& \leq (2 (1 + \norm{A}) L)^n \,.\end{split}
\end{equation}
where the second inequality counts the number of paths and uses that each path can contain at most $n$ entries strictly above the diagonal.
\end{proof}

\subsection{Probability Background}

Next, we recall the definition of $\ell_{4\to 2}$-hypercontractivity for distributions. 

\begin{definition}[Hypercontractivity]
\label{def:hypercontractivity}
We say a distribution $\mathcal{D}$ on $\R^n$ is $(4,2,K)$-hypercontractive if for any vector $v \in \R^n$, we have $\EE_{x\sim \calD}[\la v, x \ra^4] \leq K \cdot \EE_{x\sim\calD}[\la v,x \ra^2]^2$.
\end{definition}

We note that \textit{hypercontractivity} is a very mild assumption on the concentration behavior of $4$-th moments of a distribution, and several well-studied families of distributions, including all sub-Gaussian, sub-Exponential and log-concave distributions satisfy this assumption.

\begin{lemma}[Linear Transform of a Hypercontractive Distribution]
\label{claim:sum-hypercontractive}
Let $\{u_i\}_{i \in[t]}$ be $t$ iid samples from a distribution $\calD$ that is $(4,2,K)$-hypercontractive (see Definition~\ref{def:hypercontractivity}). Then, for any matrices $M_1, \dots , M_t \in \R^{m \times p}$, the random variable $\sum_{i \in [t]}  M_i u_i + \dots + M_tu_t$ is $(4,2,K)$-hypercontractive.
\end{lemma}
\begin{proof}
Fix a vector $v \in \R^m$.  Then we have
\begin{equation*}
\begin{split}
\EE[ \la v, M_1u_1 + \dots + M_tu_t\ra^4 ] &= \sum_{1 \leq i \leq t}\EE[ (v^TM_iu_i)^4] + 6\sum_{1 \leq i < i' \leq t} \EE[ (v^TM_iu_i)^2]\EE[ (v^TM_{i'}u_{i'})^2] \\ &\leq K \sum_{1 \leq i \leq t}\EE[ (v^TM_iu_i)^2]^2 + 6\sum_{1 \leq i < i' \leq t} \EE[ (v^TM_iu_i)^2]\EE[ (v^TM_{i'}u_{i'})^2] \\ &\leq K \left( \sum_{1 \leq i \leq t} \EE[(v^TM_iu_i)^2 ] \right)^2 \\ &= K\EE[ \la v, M_1u_1 + \dots + M_tu_t\ra^2 ]^2 \,.
\end{split}
\end{equation*}
By definition, this means that $M_1u_1 + \dots + M_tu_t$ is $(4,2,K)$-hypercontractive, as desired.
\end{proof}

Next, we obtain a weak anti-concentration bound via a Paley–Zygmund like inequality:

\begin{lemma}[Weak Anti-Concentration via Hypercontractivity]\label{claim:anti-concentration}
Let $z$ be a real-valued random variable such that $\EE[z] = 0, \EE[z^2] \geq 1, \EE[z^4] \leq K$ for some constant $K \geq 3$.  Then for any real number $\beta$, 
\[
\Pr[ |z - \beta| \leq 0.1] \leq 1 - \frac{1}{10K} \,.  
\]
\end{lemma}
\begin{proof}
Clearly we must have $K \geq 1$.  Assume for the sake of contradiction that the desired inequality is false.  Without loss of generality we have $\beta \geq 0$.  First consider the case where $\beta \geq 0.3$.  Let $p$ be the probability that $z \leq 0.2$.  We must have $p \leq 1/(10 K)$.  Furthermore, since $\EE[z] = 0$, we must have
\[
0 = \EE[z] = p \EE[z \big| z \leq 0.2]  + ( 1 - p) \EE[z \big| z > 0.2] \geq p \EE[z \big| z \leq 0.2] + 0.2(1 - p)
\]
which rearranges as
\[
\EE[z \big| z \leq 0.2] \leq \frac{-0.2(1 - p)}{p} \,.
\]
Thus, by Jensen's inequality (since $z^4$ is convex), this implies that 
\[
\EE[z^4] \geq p \EE[z^4 \big| z \leq 0.2] \geq  p\left(\frac{0.2(1 - p)}{p} \right)^4 > K
\]
which is a contradiction.  Now it remains to consider the case where $\beta \leq 0.3$.  Then let $q$ be the probability that $|z - \beta| \leq 0.1$.  We have
\begin{align*}
\EE[z^2] &= q\EE\left[z^2 \big|\; |z - \beta| \leq 0.1 \right] + (1 - q) \EE\left[z^2 \big|\; |z - \beta| > 0.1\right] \\ &\leq q (\beta + 0.1)^2 + (1 - q) \EE\left[z^2 \big|\; |z - \beta| > 0.1 \right] \,.
\end{align*}
Thus, since $\EE[z^2] \geq 1$, we must have
\[
\EE\left[z^2 \big|\; |z - \beta| > 0.1 \right] \geq \frac{0.8}{1 - q} \,.
\]
Thus, by convexity, we must have
\[
\EE[z^4] \geq (1 - q) \cdot \left( \frac{0.8}{1 - q}  \right)^2  = \frac{0.64}{1 - q}
\]
and combining with the fact that $\EE[z^4] \leq K$, we deduce that  $q \leq 1 - \frac{1}{10K}$ and we are done.
\end{proof}

\subsection{Linear Dynamical Systems Background}

Next, we establish some basic definitions and identities that we utilize throughout. We begin with the definition of the Markov parameters of a LDS.

\begin{definition}[Markov Parameters]
\label{def:markov-params}
Given a linear dynamical system, $\calL\Paren{ A,B, C,D}$, and an integer $T\geq 1$, the Markov Parameter matrix $G\in \R^{m \times Tp} $ is defined as the following block matrix:
\begin{equation*}
    G = \begin{bmatrix}
    D &  CB &  CAB &  \ldots & CA^{T-2}B
    \end{bmatrix}. 
\end{equation*}
\end{definition}

It will be important to consider linear combinations of the observations $y_t$.  In particular, for different integers $k$, we will consider linear combinations of the form $\hat{y}_{t+k} = y_{t+k} - \sum_{j =1}^s \alpha_j y_{t-j}$ where $\alpha_1, \dots , \alpha_s \in \R^{m \times m}$ are $m \times m$ matrices.  To ease notation, it will be useful to consider the following matrix polynomial.  

\begin{definition}[Matrix Polynomial]\label{def:matrix-polynomial}
For $\alpha = (\alpha_1, \dots , \alpha_s)$ where $\alpha_1, \dots , \alpha_s \in \R^{m \times m}$ are matrices, define the matrix polynomial $F_{\alpha,k}: \R^{n \times n} \rightarrow \R^{m \times n}$
\[
F_{\alpha,k}(X)  = CX^{s + k} - \alpha_1 CX^{s-1} - \dots - \alpha_s C \,.
\]
Also for $i = 0,1, \dots , k+s$, let
 \[F_{\alpha,k}^{(i)}(X) = \begin{cases} 
      CX^{i} &  0 \leq i \leq k\\
      CX^{i} - \sum_{j = 1}^{i-k} \alpha_jCX^{i-k-j}  & k+1 \leq i \leq k+s 
   \end{cases}\]

\end{definition}

We now have the following identities.

\begin{fact}[Algebraic Identities for LDS's]\label{fact:formula}

Let  $\calL\Paren{A,B,C,D }$ be a Linear Dynamical System (see Definition \ref{model:LDS}). Then, for any $t \in \mathbb{N}$, 
\begin{equation*}
    y_t = \sum_{i = 1}^{t} \Paren{  C A^{i-1} B   u_{t-i}  + C A^{i-1}w_{t-i} } + CA^{t}x_0 +  Du_t + z_{t}  ,
\end{equation*}
Further, given $k \in \mathbb{N}$ and $\alpha_1, \alpha_2, \ldots , \alpha_s \in \mathbb{R}^{m \times m}$, let $\hat{y}_{t+k} = y_{t+k} - \sum_{j =1}^s \alpha_j y_{t-j}$. Then, 
\begin{equation*}
\begin{split}
    \wh{y}_{t + k} = & \left( z_{t + k}  - \sum_{ j = 1}^s \alpha_j z_{t-j} \right) + \left( D u_{t+k}   - \sum_{ j = 1 }^s \alpha_j Du_{t-j} \right) +   \Paren{  F_{\alpha, k}(A) A^{t - s}  } x_0   \\ &  + \sum_{i = 1}^{k+s}  F_{\alpha,k}^{(i-1)}(A)(Bu_{t+k- i} + w_{t+k  - i}) + \sum_{i = k+s + 1}^{t + k } F_{\alpha, k }(A) A^{i  - (k+s + 1)}(Bu_{t+k- i} + w_{t+k  - i}).
\end{split}
\end{equation*}
\end{fact}

\begin{proof}
Plugging in the recursive definition of $x_t$, we have, 
\begin{equation}
\begin{split}
y_{t} &= C A x_{t-1} + CB u_{t-1} + C w_{t-1}  + D u_t + z_t \\
& = CA^2 x_{t-2} + CAB u_{t-2} + CA w_{t-2}  +  CB u_{t-1} + C w_{t-1}  + D u_t + z_t\\
& \hspace{0.2in} \vdots\\
& = Du_t + z_t + \sum_{i = 1}^{t} CA^{i-1} B u_{t-i}  + \sum_{i= 1}^{t} CA^{i-1} w_{t-i} + CA^{t}x_0
\end{split}
\end{equation}

Next, recall $\hat{y}_{t+k} = y_{t+k} - \sum_{j =1}^s \alpha_j y_{t-j}$ and thus 

\begin{equation}
\label{eqn:hat-y-expansion}
\begin{split}
    \hat{y}_{t+k} & =  \Paren{ z_{t+k} - \sum_{j =1}^s \alpha_j z_{t-j} } + \Paren{ D u_{t+k} - \sum_{j =1}^s \alpha_jD u_{t-j} } + \underbrace{ \Paren{ C A^{t+k} - \sum_{j  = 1}^s \alpha_j CA^{t-j} } x_0 }_{\eqref{eqn:hat-y-expansion}.(1)}  \\
    & \hspace{0.2in}+ \underbrace{ \sum_{ i =1 }^{t+k} \Paren{ CA^{i-1} B u_{t+k-i}  + CA^{i-1} w_{t+k-i}  }  - \sum_{j = 1}^s \alpha_j \Paren{  \sum_{ i = 1  }^{t-j } \Paren{ CA^{i-1} B u_{t-j-i}  + CA^{i-1} w_{t-j-i}  }  }}_{\eqref{eqn:hat-y-expansion}.(2)} \\
    % & = \Paren{ z_{t+k} - \sum_{j = 1}^s \alpha_j z_{t-j} } +  \Paren{ D u_{t+k} - \sum_{j =1}^s \alpha_j D u_{t-j} } + \Paren{  F_{\alpha, k}(A) A^{t - s}  } x_0  \\ & \quad + \sum_{i = 1 }^{k+1} \Paren{ CA^{i-1} B u_{t+k-i}  + CA^{i-1} w_{t+k-i}  } \\ & \quad +   \sum_{i = k+2}^{ k + s} \Paren {\Paren{  C A^{i-1} - \sum_{j = 1 }^{i-k - 1}\alpha_j CA^{i - (k+j + 1)}   } B u_{t+k-i}   + \Paren{  C A^{i-1} - \sum_{j = 1 }^{i-k - 1}\alpha_j CA^{i - (k+j + 1)}   } w_{t+k-i} }  \\ &  \quad + \sum_{i = k+s + 1}^{t+k } \Paren {\Paren{  C A^{i-1} - \sum_{j = 1 }^{s}\alpha_j CA^{i - (k+j + 1)}   } B u_{t+k-i}   + \Paren{ C A^{i-1} - \sum_{j = 1 }^{s}\alpha_j C A^{i - (k+j + 1)}   } w_{t+k-i} }
    % \\ &= \Paren{ z_{t+k} - \sum_{j = 1}^s \alpha_j z_{t-j} } +  \Paren{ D u_{t+k} - \sum_{j =1}^s \alpha_j D u_{t-j} } + \Paren{  F_{\alpha, k}(A) A^{t - s}  } x_0
    % \\ & \quad + \sum_{i = 1}^{k+s}  F_{\alpha,k}^{(i-1)}(A)(Bu_{t+k- i} + w_{t+k  - i}) + \sum_{i = k+s + 1}^{t + k } F_{\alpha, k }(A) A^{i  - (k+s + 1)}(Bu_{t+k- i} + w_{t+k  - i}) \,.
\end{split}
\end{equation}
First, observe,
\begin{equation}
  \eqref{eqn:hat-y-expansion}.(1) = \Paren{  F_{\alpha, k}(A) A^{t - s}  } x_0 . 
\end{equation}
Next, we split the terms in \eqref{eqn:hat-y-expansion}.(2) into the following ranges: $[1, k+s]$ and $[k+s+1, t+k]$, and expanding out each term, we have
\begin{equation}
\label{eqn:f-alpha-terms}
    \begin{split}
        \eqref{eqn:hat-y-expansion}.(2)  & = \sum_{i = 1}^{ k + s} \Paren { \underbrace{ \Paren{  C A^{i-1} - \sum_{j = 1 }^{i-k - 1}\alpha_j CA^{i - (k+j + 1)}   } }_{\eqref{eqn:f-alpha-terms}.(1)} \Paren{ B u_{t+k-i} + w_{t+k-i} }  }      \\ &  \quad + \sum_{i = k+s + 1}^{t+k } \Paren { \underbrace{ \Paren{  C A^{i-1} - \sum_{j = 1 }^{s}\alpha_j CA^{i - (k+j + 1)}   } }_{\eqref{eqn:f-alpha-terms}.(2)} \Paren{ B u_{t+k-i} + w_{t+k-i} }   }
    \end{split}
\end{equation}
Recalling the definition of $F_{\alpha,k}^{i}$ and $F_{\alpha, k}$,
\begin{equation}
\begin{split}
    \eqref{eqn:f-alpha-terms}.(1)  & = F_{\alpha,k}^{i-1}\\
    \eqref{eqn:f-alpha-terms}.(1) & =  F_{\alpha,k} A^{i - (k+s+1)},
\end{split}
\end{equation}
and therefore, 
\begin{equation*}
    \eqref{eqn:hat-y-expansion}.(2) =  \sum_{i = 1}^{k+s}  F_{\alpha,k}^{(i-1)}(A)(Bu_{t+k- i} + w_{t+k  - i}) + \sum_{i = k+s + 1}^{t + k } F_{\alpha, k }(A) A^{i  - (k+s + 1)}(Bu_{t+k- i} + w_{t+k  - i}) 
\end{equation*}
Combining all the terms together, we can conclude
\begin{equation*}
\begin{split}
    \hat{y}_{t+k} = &  \Paren{ z_{t+k} - \sum_{j = 1}^s \alpha_j z_{t-j} } +  \Paren{ D u_{t+k} - \sum_{j =1}^s \alpha_j D u_{t-j} } + \Paren{  F_{\alpha, k}(A) A^{t - s}  } x_0 \\
    & + \sum_{i = 1}^{k+s}  F_{\alpha,k}^{(i-1)}(A)(Bu_{t+k- i} + w_{t+k  - i}) + \sum_{i = k+s + 1}^{t + k } F_{\alpha, k }(A) A^{i  - (k+s + 1)}(Bu_{t+k- i} + w_{t+k  - i}) ,
\end{split}
\end{equation*}
as desired.
\end{proof}

Next, we observe that the \textit{cross-covariance} between the control input and the observation is an unbiased estimator of the Markov parameters. 

\begin{fact}[Cross-Covariance of Control and Observation]\label{fact:markov-params-formula}
For any $t,k \in \mathbb{N}$, and any $0 \leq j \leq k$, we have
    \begin{equation}
        \expecf{}{\hat{y}_{t+j} u_t^\top } =  \begin{cases} D \text{ if } $j = 0$ \\
        C A^{ j - 1} B \text{ otherwise}
        \end{cases}
    \end{equation}
\end{fact}
\begin{proof}
We use the formula in Fact~\ref{fact:formula} and the independence of the $u_i,w_i,z_i$.  When $j = 0$, we immediately have 
\[
\EE[\hat{y}_{t} u_t^\top ] = \EE[Du_t u_t^\top] = D \,.
\]
When $j > 0$, we have
\[
\EE[\hat{y}_{t+j} u_t^\top ]  = \EE[F_{\alpha,k}^{(j-1)}(A)B u_t u_t^\top  ] = F_{\alpha,k}^{(j-1)}(A)B \,.
\]
Next, since $j \leq k$, by definition we have $F_{\alpha,k}^{(j-1)}(A) = CA^{j-1}$ and we are done.
\end{proof}

In light of the above, we make the following definition. 
\begin{definition}\label{def:estimator-means}
 For an integer $j \geq 0$, we define the matrix $X_j$ as 
\[
X_j =   \begin{cases} D \text{ if } $j = 0$ \\
C A^{ j - 1} B \text{ otherwise}
\end{cases} \,.
\]
\end{definition}
\noindent Of course, we have $X_j = \EE[\hat{y}_{t + j} u_t^\top ] $ by the previous fact.

We also require the following straight-forward consequences of the bounded condition number assumption on the observability and controlability matrices from Section~\ref{sec:formal-setup}:

% \textcolor{red}{A:can we move these simple claims to the prelims of LDS's section? that way this subsection is just the assumptions.}
\begin{claim}\label{claim:eigenvalue-bound}
Given an integer $s$, let $O_s, Q_s$ be the observability and controlability matrices from definition \ref{def:observability} and \ref{def:controllability}. Then,  $\sigma_{\min}(O_s) \leq \sqrt{s}\norm{C}$ and  $\sigma_{\min}(Q_s) \leq \sqrt{s}\norm{B}$.
\end{claim}
\begin{proof}
Let $v$ be an eigenvector of $A$ and say $Av = \lambda v$.  Note that $|\lambda| \leq 1$.  Thus,
\[
\norm{O_sv} \leq \sqrt{\norm{C v}^2 ( 1 + \lambda^2 + \lambda^4 + \dots + \lambda^{2s-2}) }  \leq \sqrt{s} \norm{C} \norm{v} \,.
\]
A similar argument works for $Q_s$.
\end{proof}

\begin{claim}\label{claim:boundA} Given integers $t> s$, let $\sigma_{\max}(O_{t})/\sigma_{\min}(O_s) \leq \kappa$.  Then for any integer $k > 0$,
\[
\norm{A^k}_F \leq (\sqrt{n} \kappa)^{ k/(t - s)  } \,. 
\]
The same holds with $O$ replaced with $Q$.
\end{claim}
\begin{proof}
Note that $\norm{O_s A^{t-s}}_F \geq \sigma_{\min}(O_s) \norm{A^{t-s}}_F$.  On the other hand,
\[
\norm{O_s A^{t-s}}_F \leq \norm{O_t}_F \leq \sqrt{n} \norm{O_t} \leq \sqrt{n}\kappa \sigma_{\min}(O_s) \,.
\]
Thus, $\norm{A^{t-s}}_F \leq \sqrt{n} \kappa$ and this immediately implies the desired inequality.
\end{proof}

\section{Algorithm}

Our algorithm follows the outline described in Section~\ref{sec:technical-overview}.  We first learn a stabilizing transformation that transforms $y_t \to \hat{y}_t$ (see Algorithm~\ref{algo:stabilizing}).  Note that in the general case, when $C$ is $m \times n$ (as opposed to $1 \times n$ as described in Section~\ref{sec:technical-overview}), the coefficients $\alpha_1, \dots , \alpha_s$ are $m \times m$ matrices.  Once we have these matrices, we empirically estimate the Markov parameters (recall Fact~\ref{fact:markov-params-formula}).  Finally, once we obtain estimates for the Markov parameters, we extract the system parameters via the Ho-Kalman algorithm (see Algorithm~\ref{algo:robust-ho-kalman}).

% \textcolor{red}{A: intuitive explanations of the two algorithms. In particular, highlight that $y_t \to \hat{y}_t$ does not change the expectation, but helps bound the variance directly, and $\calC_\alpha$ truly requires enforcing the constraints on shifted observations, instead of every single observation. }

Before we describe our algorithm formally, we introduce a few definitions and notational simplifications.  Throughout this section, we will use $\eps$ to denote the desired accuracy and $\delta$ to be a parameter for the failure probability.  We will also make the simplification that $x_0 = 0$.  This is allowed because we can absorb the distribution of $x_0$ into the distribution for $w_0$ i.e. $\calD_w \leftarrow A\calD_0 + \calD_w$ which is still polynomially bounded in the system parameters.  As mentioned in Remark~\ref{rem:not-iid}, we do not need the $w_i$ to have identical distributions, just that they are independent and bounded variance.

\begin{mdframed}
  \begin{algorithm}[Learning a Linear Dynamical System]
    \label{algo:learning-lds}\mbox{}
    \begin{description}
    \item[Input:] $T$ observations $\Set{y_1, y_2, \ldots , y_T}$ and the corresponding control inputs $\Set{u_1, u_2, \ldots , u_T}$  generated from a Linear Dynamical System $\calL\Paren{A,B,C,D}$ satisfying the assumptions in Section~\ref{sec:formal-setup}.
    
    \item[Input:]   $s \in \mathbb{N}$ such that the observability and controllability matrices satisfy the condition number bounds in Definition~\ref{def:system-assumptions},   accuracy parameter $0 < \epsilon < 1$, and failure probability parameter $0 < \delta< 1$.
    
    \item[Operation:]\mbox{}
    \begin{enumerate}
    \item \textbf{Stabilizing the System:} Run Algorithm \ref{algo:stabilizing} to obtain coefficient matrices $\alpha_1, \ldots, \alpha_s$. 
    \item \textbf{Estimating the Markov Parameters:} Set $k = 10s$.  For all $j = 0,1, \dots , k$ %\textcolor{red}{A: how far do we need to go to estimate markov parameters for Ho-Kalman?} 
    \begin{enumerate}
    \item For all $t \in [T]$ with $t > k + s$, compute $\hat{y}_{t} = y_{t} - \sum_{i=1}^{s} \alpha_i y_{t-k - i}$.  
    \item Compute $\hat{X}_j = \frac{1}{T - k - s } \sum_{ t = k + s + 1 }^{T } \hat{y}_{t} u_{t- j}^\top$. 
    \end{enumerate}

    \item \textbf{Robust Ho-Kalman:} Run Ho-Kalman on $\hat{G} = [\hat{X}_0, \hat{X}_1, \dots , \hat{X}_{2s}]$ to obtain estimates $\hat{A},\hat{B},\hat{C},\hat{D}$
    %Let $\hat{D} = X_0$ and construct the block matrix
    %\[
    %M = \begin{bmatrix}
    %\hat{X}_1 & \hat{X}_2 & \dots & \hat{X}_{s + 1} \\
    %\hat{X}_2 & \hat{X}_3 & \dots & \hat{X}_{s + 2} \\
    %\vdots & \vdots & \ddots & \vdots \\
    %\hat{X}_{s} & \hat{X}_{s + 2} & \dots & \hat{X}_{2s}
    %\end{bmatrix}
    %\]
    %\textcolor{red}{A:what exactly does Ho-Kalman require?}
    \end{enumerate}
     \item[Output:] $\hat{A}, \hat{B}, \hat{C}, \hat{D}$ satisfying guarantees of \pref{thm:main-learning-lds}
     %\textcolor{red}{A:fill here with the right guarantees coming out of Ho-Kalman}\textcolor{blue}{Allen: can we just state the guarantees in the theorem statements but not in the algorithm boxes?}\abnote{yep!}
    \end{description}
  \end{algorithm}
\end{mdframed}

In the definition below, we condense all of the system parameters into a set $\calS$ so that we don't need to list out the full set of parameters in future computations.  Recall Definition~\ref{def:distributional-assumptions} where $K$ is the hypercontractivity parameter for the input distribution (and will be defined as such throughout this section) and $\sigma_w, \sigma_z$ are upper bounds on the variances of the noise distributions.

\begin{definition}\label{def:parameters}
Let $\calS$ denote the set of parameters 
\[
\{\norm{A}, \norm{B}, \norm{C}, \norm{D}, m,n,p, s, \kappa , K, 1/\delta, \sigma_w, \sigma_z \} \,.
\]
We will write $\poly(\calS)$ for a quantity that depends polynomially on these parameters.
\end{definition}
Note that we will be more explicit about dependencies on the accuracy $\eps$ so it is not included in the definition of $\calS$.

Now we define a constraint system for a convex program that is at the core of our algorithm.  Let $\eps$ be the desired accuracy, and $k = 10s$.  Throughout this section, we will treat $k$ as fixed.  Choose sufficiently large polynomials $P_0,P_1, P_2$ in terms of the parameters in $\calS$ such that $P_0 \ll P_1 \ll P_2 $ and let  $L = P_2 \log^2 (1/\eps) /\eps^2$.

\begin{definition}\label{def:program-v2}
We define the constraint system $\calC_{\alpha}$  for matrices of coefficients $\alpha_1, \dots , \alpha _s \in \R^{m \times m}$ as follows.  We define $\hat{y}_{t+k} = y_{t+k} - \alpha_1 y_{t-1} - \dots - \alpha_s y_{t-s}$.  We also enforce the following constraints
\begin{enumerate}
    \item $\norm{\alpha_i}_F \leq P_0$ for all $i = 1,2, \dots , s$
    \item For all $i \in [100 s n m^2 K \log L]$, we have
    \[\|\hat{y}_{iL + k} \|\leq P_1 \log(1/\eps) \]
    %\item For all $r$ with $k + s \leq r \leq P_2$, we have
    %\[\frac{1}{P_3} \sum_{t=1}^{P_3} u_{t} \hat{y}_{t+r}^T \leq \eps \]
\end{enumerate}
\end{definition}

\begin{mdframed}
  \begin{algorithm}[Stabilizing the System]
    \label{algo:stabilizing}\mbox{}
    \begin{description}
    \item[Input:] $T$ observations $\Set{y_1, y_2, \ldots , y_T}$ from a single trajectory and the corresponding control inputs $\Set{u_1, u_2, \ldots , u_T}$  generated from a Linear Dynamical System $\calL\Paren{A,B,C,D}$ satisfying the assumptions in Section~\ref{sec:formal-setup}
    
    \item[Input:]   $s \in \mathbb{N}$ such that the observability and controllability matrices satisfy the condition number bounds in Definition~\ref{def:system-assumptions},   accuracy parameter $0 < \epsilon < 1$, and failure probability parameter $0 < \delta< 1$.
    
    \item[Operation:]\mbox{}
    \begin{enumerate}
    \item Let $P_0, P_1, P_2$ be sufficiently large polynomials in the parameters $m, n, p,s,k,\kappa$ and $K$ such that $P_0 \ll P_1 \ll  P_2$. Let $L = P_2\log^2(1/\epsilon)/\epsilon^2$. Let $S = \Set{1,2, \ldots }$. 
    \item Solve the following system in the matrix variables $\alpha_1, \alpha_2, \ldots \alpha_s \in \mathbb{R}^{m \times m}$:
    \begin{equation*} \calC_{\alpha}=
  \left \{
    \begin{aligned}
      & \forall j \in [s]
      & \Norm{\alpha_j}_F^2 
      & \leq  P_0\\
      & \forall i \in [100 snm^2 K \log(L)]
      &  \Norm{ y_{iL +  k }  - \sum_{j \in [s]}  \alpha_j\cdot  y_{iL -j } }_2^2 
      & \leq  P_1 \log(1/\epsilon) \\
    \end{aligned}
  \right \}
\end{equation*}
    \end{enumerate}
    \item[Output:] Matrices $\alpha_1, \alpha_2, \ldots \alpha_s$ obtained above.
    %satisfying the guarantees of Theorem \ref{}. 
    \end{description}
  \end{algorithm}
\end{mdframed}

 Our main theorem is that Algorithm~\ref{algo:learning-lds} runs with polynomial time and sample complexity and outputs estimates $\hat{A},\hat{B},\hat{C},\hat{D}$ that are $\eps$-close to the true system parameters up to a global similarity transformation (which is always necessary).
 
\begin{theorem}[Learning a Linear Dynamical System]
\label{thm:main-learning-lds}

Given $0 < \epsilon, \delta < 1$, an integer $s$, and trajectory length $T = \Omega\Paren{\poly(\mathcal{S}) \cdot \frac{\log^3(\frac{1}{\epsilon})}{\epsilon^4}}$, and the corresponding observations and inputs $\Set{y_i, u_i}_{i \in [T]}$ , from a linear dynamical system $\calL\Paren{ A,B,C,D}$, satisfying the assumptions in Section~\ref{sec:formal-setup}, Algorithm \ref{algo:learning-lds} outputs estimates $\hat{A}, \hat{B}, \hat{C} , \hat{D}$ such that with probability at least $1-\delta$, there exists a similarity tranform $U$ satisfying 

\begin{align*}
    \max \Paren{  \Norm{A - U^{-1} \hat{A}  U } ,   
    \Norm{C - \hat{C}U } , 
     \Norm{B - U^{-1} \hat{B} } ,   
    \Norm{D - \hat{D}}  } \leq  \epsilon , 
\end{align*}
Further, Algorithm \ref{algo:learning-lds} runs in $\poly(\mathcal{S},\frac{1}{\epsilon})$ time. 
\end{theorem}

The overall structure of our approach is to first estimate the Markov parameters (see Definition \ref{def:markov-params}). We then recover the estimates of the system parameters by setting up a Generalized eigenvalue problem and using the Ho-Kalman Algorithm (see for example Theorem 5.3 in~\cite{oymak2019non}).   

The key techincal theorem we obtain for learning each block matrix in the Markov Parameter matrix is as follows: 

\begin{theorem}[Learning the Markov Parameters]
\label{thm:learning-markov}
Given $0 <  \epsilon, \delta < 1$, an integer $s$, and trajectory length $T = \Omega\Paren{\poly(\mathcal{S}) \cdot \frac{\log^3(\frac{1}{\epsilon})}{\epsilon^2}}$, with the corresponding observations and inputs $\Set{y_i, u_i}_{i \in [T]}$, from a linear dynamical system $\calL\Paren{ A,B,C,D}$, satisfying the assumptions in Section~\ref{sec:formal-setup}, Algorithm \ref{algo:stabilizing} outputs $\alpha_1, \alpha_2, \ldots, \alpha_s$ in time $\poly(\mathcal{S},1/\epsilon)$ such that with probability at least $1-\delta$, for all $0 \leq j \leq k$
\begin{equation*}
    \Norm{ \frac{1}{T - k - s} \sum_{ t=  k + s + 1}^{T }  \Paren{ y_{t} - \sum_{i \in [s]} \alpha_i y_{t-k - i}  }u_{t - j}^\top  - X_j }_F^2 \leq \epsilon^2  
\end{equation*}
where recall $X_j$ is defined in Definition~\ref{def:estimator-means}.
\end{theorem}

Once we have proven Theorem~\ref{thm:learning-markov}, we can combine it with a result from \cite{oymak2019non} for estimating the system parameters from the Markov Parameters using Ho-Kalman.  See Section~\ref{sec:markov-params-to-sys-params} for details.  The main technical work of this paper is in proving Theorem \ref{thm:learning-markov}, which we focus on in Section~\ref{sec:learn-markov}.

\iffalse

\begin{lemma}[\cite{oymak2019non}]
For observability and controllability matrices $\calO_s, \calQ_s$ that are rank $n$, the Ho-Kalman algorithm applied to $\hat{G}$ produces estimates $\hat{A},\hat{B}$, and $\hat{C}$ such that 
there exists unitary matrix $T \in \R^{n \times n}$ such that 

\[\max\{\|C - \hat{C} T\|_F, \| B - T^* \hat{B}\|_F\} \leq 5\sqrt{n \| G - \hat{G}\|}
\]
 and 
 \[\|A - T^* \hat{A}T\|_F \leq \frac{\sqrt{n \|G - \hat{G}\|} \| H\| }{\sigma_{min}^{3/2}(H^{-})}\]
 
 and 
 \[\Norm{D - \hat{D}}_F \leq \sqrt{n}\Norm{G - \hat{G}}\]
 where in the above
 \[
 G = [D, CB, CAB, ..., CA^{2s-1}B]  \,.
 \]
\end{lemma}

\begin{proof}[Proof of Theorem \ref{thm:main-learning-lds} ]

\textcolor{red}{A: proof sketch: } Invoke Theorem \ref{thm:learning-markov}, union bound over all blocks. Invoke Corollary \ref{cor:generalized-eigen-decom}. 
Add running time and sample complexity bounds.
\end{proof}
\fi

% We stabilize the system via the following key theorem:

% \begin{theorem}[Learning the Coefficient Matrix Polynomial]

% \end{theorem}

% \begin{lemma}[Separation Oracle]
% \textcolor{red}{A: fill in why we can solve the system defined in Algorithm \ref{algo:stabilizing}}.

% \end{lemma}

\vspace{0.5in}

% Our estimator for $CA^k B$ is 
% \[\frac{1}{P_2} \sum_{t=1}^{P_2} u_t \hat{y}_{t+k}^T \approx CA^k B\]
% where $\hat{y}_{t+k} = y_{t+k} - a_1 y_{t-1} - a_2 y_{t-2} - ... - a_s y_{t-s}$.  The overarching idea is simple.  Since, the first $k$ terms of $\hat{y}_{t+k}$ are of the form $Cu_{t+k} + CAu_{t+k-1} + ... + CA^k u_{t} $ we can   

%   For coefficients $a_1,...,a_s$ satisfying the following constraints.  

% \begin{enumerate}
%     \item (Residuals are Small) \[\|\hat{y_t} \|\leq \poly_1(k,\text{params})\]
%     \item (Residuals are Uncorrelated with the Distant Past) for all $r \in [1,t-1]$
%     \[\frac{1}{P_2} \sum_{t=1}^{P_2} u_{t-r} \hat{y}_{t+k}^T \leq \poly_2( \text{params})\]
% \end{enumerate}

% The exact quantities for $\poly_1$ and $\poly_2$ are as follows.  We are building off the feasible $CA^{k+1} = a_1CA^k + a_2CA^{k-1} + ... + a_k C$

% \subsection{Program Formulation}

\section{Analysis of Algorithm \ref{algo:stabilizing} }\label{sec:learn-markov}
In this section, we analyze Algorithm \ref{algo:stabilizing} and prove Theorem \ref{thm:learning-markov}.  Since we treat $k$ as fixed throughout this section, we will write $F_{\alpha}(A)$ for $F_{\alpha , k}(A)$ (recall Definition~\ref{def:matrix-polynomial}).  First, we need the following basic observation.

\begin{claim}[Uniform bounds on the control and noise]\label{claim:basic-boundedness}
Let $P_0, P_1, P_2$ be sufficiently large polynomials in the parameters in $\calS$ such that $P_0 \ll P_1 \ll  P_2$. Let $L = P_2\log^2(1/\epsilon)/\epsilon^2$.  With probability $1 - 0.1\delta$, the following events all hold:
\begin{enumerate}
    \item For all $i \in [100 s n m^2 K \log L]$ and integers $-10(k+s) \leq c  \leq 10(k + s)$, we have
    \begin{equation*}
    \begin{split}
    \norm{u_{iL + c}}, \norm{w_{iL + c}}, \norm{z_{iL + c}} \leq P_0 \log(1/\eps) \,.    
    \end{split}
    \end{equation*}
    \item For all $t \leq P_2^2 \log^3 (1/\eps)/\eps^2 $, we have
    \begin{equation*}
    \begin{split}
    \norm{u_{t}}, \norm{w_{t}}, \norm{z_{t}} \leq L P_0 \,.    
    \end{split}
    \end{equation*}
\end{enumerate}
\end{claim}
\begin{proof}
Note that from the covariance bounds on $u_t,w_t,z_t$ (and using Markov's inequality), we have that for a fixed index $t$ and any parameter $\alpha$,
\begin{align*}
&\Pr[ \norm{u_t} \geq \alpha ] \leq \frac{n}{\alpha^2} \\
&\Pr[ \norm{w_t} \geq \alpha ] \leq \frac{n\sigma_w^2}{\alpha^2} \\
&\Pr[ \norm{z_t} \geq \alpha ] \leq \frac{n \sigma_z^2}{\alpha^2} \,.
\end{align*}
Now to prove the first statement, note that we only need to union bound over 
\[
O(100(k + s) s n m^2 K \log L) = \poly(\calS) \log(1/\eps)
\]
variables so as long as we choose $P_0$ sufficiently large we get that the statement holds with at least $1 - 0.05\delta$ probability.  The proof of the second statement is similar except we union bound over more variables. 
\end{proof}

We begin by establishing the feasibility of the constraint system, $\calC_{\alpha}$, as defined in Algorithm \ref{algo:stabilizing}: 

\begin{lemma}[Feasibility of the Constraint System]
\label{lem:feasibility}
Assume we are given $0 < \epsilon, \delta < 1$, and $T = \Omega\Paren{\poly(\mathcal{S}) \cdot \frac{\log^3(\frac{1}{\epsilon})}{\epsilon^2}}$  observations $\Set{y_i, u_i}_{i \in [T]}$  from a single trajectory of a linear dynamical system $\calL\Paren{ A,B,C,D}$, satisfying the assumptions in Section~\ref{sec:formal-setup}.  Then, as long as the events in Claim~\ref{claim:basic-boundedness} hold, the constraint system $\calC_{\alpha}$ is feasible.

\iffalse
let $P_0= $, $P_1 = $,  $P_2 =$ and $L= $.  Then, with probability at least $1-\delta$, the constraint system 
\begin{equation*}
    \calC_{\alpha}=
  \left \{
    \begin{aligned}
      & \forall j \in [s]
      & \Norm{\alpha_j}_F^2 
      & \leq  P_0\\
      & \forall i \in [100 snm^2 K \log(L)]
      &  \Norm{ y_{iL +  k }  - \sum_{j \in [s]}  \alpha_j\cdot  y_{iL -j } }_2^2 
      & \leq  P_1 \log(1/\epsilon) \\
    \end{aligned}
  \right \}
\end{equation*}
is feasible. 
\fi
\end{lemma}

\begin{proof}
Consider the matrix $CA^{k+s}$.  By Claim~\ref{claim:boundA}, we have
\[
\norm{CA^{k+s}}_F \leq (\sqrt{n}\kappa)^{(k + s)/s} \norm{C} = (\sqrt{n}\kappa)^{11} \norm{C}  \,.
\]
since we defined $k = 10s$, where $s$ is the observability and controlability parameter.  Now we use the bounded condition number of the observability matrix (see Definition~\ref{def:system-assumptions}):  clearly $\sigma_{\max}(O_{2s}) \geq \norm{C} \geq 1$.  Thus, $\sigma_{\min}(O_s) \geq 1/\kappa$.  This means that there must be $m \times m$ matrices $\alpha_1, \dots , \alpha_s$ with 
\begin{equation}\label{eq:coeff-frob-bound}
\norm{\alpha_i}_F \leq \kappa \norm{CA^{k+s}}_F \leq \poly(\sqrt{n}\kappa\norm{C})
\end{equation}
such that 
\begin{equation}\label{eq:matrix-poly-is-zero}
CA^{k+s} = \alpha_1C + \alpha_2CA + \dots + \alpha_s CA^{s-1} \,.
\end{equation}
By choosing $P_0$ appropriately, these $\alpha_i$ clearly satisfy the first set of constraints in $\calC_{\alpha}$ on their Frobenius norm.  It remains to verify that these $\alpha_i$ satisfy the second constraint.  Note that \eqref{eq:matrix-poly-is-zero} implies that $F_{\alpha}(A) = 0$
By Fact~\ref{fact:formula}, we have
\begin{equation*}
\begin{split}
    &\wh{y}_{t + k} =  \left( z_{t + k}  - \sum_{ j = 1}^s \alpha_j z_{t-j} \right) + \left( D u_{t+k}   - \sum_{ j = 1 }^s \alpha_j Du_{t-j} \right)    + \Paren{  F_{\alpha}(A) A^{t - s}  } x_0   \\ & \quad + \sum_{i = 1}^{k+s}  F_{\alpha}^{(i-1)}(A)(Bu_{t+k- i} + w_{t+k  - i}) + \sum_{i = k+s + 1}^{t + k } F_{\alpha }(A) A^{i  - (k+s + 1)}(Bu_{t+k- i} + w_{t+k  - i}) \,.
\end{split}
\end{equation*}
The last sum in the above is $0$ since $F_{\alpha}(A) = 0$.  For the remaining terms, there are a total of $\poly(k+s)$ terms and each one can be upper bounded in terms of
\[
\poly \left( \max_{c \in [-(k+s), k+s]}( \norm{z_{t + c}}, \norm{u_{t + c}}, \norm{w_{t + c}}) , \max_{0 \leq i \leq k+s}(\norm{A^{k+s}}), \max_{i \in [s]}\norm{\alpha_i} , \norm{B} + \norm{C} + \norm{D} \right) \,.
\]
Next, we only need to consider $\hat{y}_{t + k}$ for $t = iL$ for $i \in [100snm^2K \log L ]$ and thus we can invoke Claim~\ref{claim:basic-boundedness} to bound the first quantity above.  We can use Claim~\ref{claim:boundA} to bound the second quantity and by definition we have $\norm{\alpha_i} \leq P_0$.  Overall, we can upper bound $\hat{y}_{iL + k}$ for all $i \in [100snm^2K \log L ]$ as $\poly(\calS, P_0) \log(1/\eps)$ and thus the solution we have constructed is feasible as long as we have chosen $P_1$ sufficiently large.
\end{proof}

\subsection{Analysis}

Next, we argue that any  $\alpha$ that is feasible for $\calC_{\alpha}$ must actually be useful for stabilizing the system.  To do this, we introduce the following potential.

\begin{definition}[Anti-Concentration Potential]
For an integer $l$ and coefficients $\alpha = (\alpha_1, \dots , \alpha_s)$, define the function 
\[
G_{\alpha , l} = \sum_{i = 0}^l \norm{F_{\alpha}(A)A^i B}_F^2 \,.
\]
\end{definition}

% \textcolor{red}{A: add inuition for why this potential is useful.} 
$G_{\alpha,L}$ is a potential measuring the variance of $\hat{y}$.  We will show that with high probability over the randomness of the $u_t,w_t, z_t$, any $\alpha$ that is feasible for $\calC_{\alpha}$ must have $G_{\alpha , l}$ be small.  

First, we express $G_{\alpha , l}$ as the variance of a random variable that naturally arises when computing $\hat{y}_t$ (using the formula in Fact~\ref{fact:formula}).
\begin{lemma}[Potential captures Variance]\label{claim:potential-is-variance}
Given $l< t \in \mathbb{N}$, consider the random variable 
\[
\gamma_{t,l} = \sum_{i = 0}^l F_{\alpha}(A) A^i B u_{t - i} \,.
\]
Then we have
\[
\textsf{Tr}\left[ \EE[ \gamma_{t,l} \gamma_{t,l}^\top] \right] = G_{\alpha , l} \,.
\]
\end{lemma}
\begin{proof}
Since the $u_i$ are independent $\EE[u_iu_j] = 0$ for $i \neq j$ and $\expecf{}{u_i u_i^\top} = I$ for all $i$,  we have
\begin{align*}
\EE[ \gamma_{t,l} \gamma_{t,l}^\top] &= \EE\left[ \left(\sum_{i = 0}^l F_{\alpha}(A) A^i B u_{t - i}\right) \left(\sum_{i = 0}^l F_{\alpha}(A) A^i B u_{t - i}\right)^\top \right] \\ &= \EE\left[ \sum_{i = 0}^l \left(F_{\alpha}(A) A^i B u_{t - i}\right) \left( F_{\alpha}(A) A^i B u_{t - i}\right)^\top \right] \\ & =  \sum_{i = 0}^l ((F_{\alpha}(A) A^i B) ((F_{\alpha}(A) A^i B)^\top \,.
\end{align*}
Now taking the trace of both sides, and using the linearity of trace, 

\begin{equation*}
\begin{split}
    \textsf{Tr}\left[ \expecf{}{\gamma_{t,l} \gamma_{t,l}^\top } \right] & = \sum_{i = 0}^{l}  \textsf{Tr} \left[ ((F_{\alpha}(A) A^i B) ((F_{\alpha}(A) A^i B)^\top \right] \\
    & = \sum_{i = 0}^l \norm{F_{\alpha}(A)A^i B}_F^2 = G_{\alpha , l}
\end{split}
\end{equation*}
as desired.
\end{proof}

Next, we show that if the potential $G_{\alpha, L}$ is large, with high probability, there must be a violated constraint in $\calC_{\alpha}$. 

\begin{lemma}\label{lem:potential-bound-v2}
Let $\alpha = (\alpha_1, \dots , \alpha_s)$ be a fixed sequence with $\norm{\alpha_i}_F \leq P_0$ and $G_{\alpha, L}  > m(100P_1 \log (1/\eps))^2$.  Now, consider a trajectory of length $T$, and the corresponding observations and inputs, $\Set{ y_i, u_i }_{i \in [T]}$, from the LDS $\calL(A,B,C,D)$ after fixing this sequence and let $\zeta$ be the event that Lemma \ref{claim:basic-boundedness} holds. Conditioned on $\zeta$,  with $1 - (1/L)^{10s n m^2}$ probability, there exists some integer $1 \leq i \leq 10^2snm^2 K \log L$ such that
\[
\norm{\hat{y}_{iL + k}} \geq 2P_1 \log(1/\eps) \,.
\]
%\textcolor{red}{A:define $K$?}
\end{lemma}
\begin{proof}

Recall the formula in Fact~\ref{fact:formula}.  We can write it as 
\[
\wh{y}_{t + k} = X + Y + Z + \sum_{i = k+s + 1}^{t + k } F_{\alpha}(A) A^{i  - (k+s + 1)}(Bu_{t+k- i} + w_{t+k  - i}) \,,
\]
where $X = \left( z_{t + k}  - \sum_{ j = 1}^s \alpha_j z_{t-j} \right)  Y= \left( D u_{t+k}   - \sum_{ j = 1 }^s \alpha_j Du_{t-j} \right) $ and $Z=\sum_{i = 1}^{k+s}  F_{\alpha}^{(i-1)}(A)(Bu_{t+k- i} + w_{t+k  - i}) $. 
Now consider $t = iL$ for some $i \in [100snm^2K \log L]$.  Note that since we are conditioning on the event in Claim~\ref{claim:basic-boundedness}, all entries of $X,Y,Z$ are bounded by some fixed polynomial in the parameters in $\calS$ and $\norm{\alpha_i}$.  Since $\norm{\alpha_i}_F \leq P_0$ for all $i$, we can choose $P_1$ larger than this polynomial i.e. we can ensure $\norm{X}, \norm{Y}, \norm{Z} \leq P_1$.  Now we can break the last sum as follows:

\begin{equation}
    \begin{split}
        \sum_{i = k+s + 1}^{t + k } &  F_{\alpha}(A) A^{i  - (k+s + 1)}(Bu_{t+k- i} + w_{t+k  - i}) \\
        & = \underbrace{ \sum_{i = k+s + 1}^{k+s + L + 1}F_{\alpha}(A) A^{i  - (k+s + 1)}Bu_{t+k- i}  }_{V_t} \\
        & \hspace{0.2in} +  \underbrace{ \sum_{i =k+s + L + 2}^{t + k } F_{\alpha}(A) A^{i  - (k+s + 1)}Bu_{t+k- i} + \sum_{i = k+s + 1}^{t + k} F_{\alpha}(A) A^{i  - (k+s + 1)}w_{t+k  - i} }_{W_t} \,.
    \end{split}
\end{equation}

Let the first term above $V_t$ and the sum of the second two terms $W_t$.  Then $V_t$ is a random variable and it follows from Lemma~\ref{claim:potential-is-variance} that the  covariance matrix $\Sigma_V$ satisfies $\textsf{Tr}(\Sigma_V) =  G_{\alpha , L}$.  Note that the random variable $V_t$ is independent of $W_t$ since they depend on disjoint sets of $u_i, w_i$.

\iffalse

\textcolor{red}{Pull this out, explain variance of this term outside}

\[\textsf{Tr}(\Sigma_V) = \textsf{Tr}\big( \EE[\big(\sum_{i = k+s}^{k+s + L}CF_{\alpha}(A) A^{i  - (k+s)}Bu_{t+k- i}\big) \big(\sum_{i = k+s}^{k+s + L}CF_{\alpha}(A) A^{i  - (k+s)}Bu_{t+k- i}\big)^T]\big)\]
Using $\EE[u_iu_j] = 0$ for all $i \neq j$ we obtain 
\[= \textsf{Tr}\big( \EE[\big(\sum_{i = k+s}^{k+s + L}CF_{\alpha}(A) A^{i  - (k+s)}Bu_{t+k- i}\big) \big(CF_{\alpha}(A) A^{i  - (k+s)}Bu_{t+k- i}\big)^T]\big)\]
Using $\EE[u_iu_i^T] = I$ we obtain 
\[= \textsf{Tr}\big( \sum_{i = k+s}^{k+s + L}\big(CF_{\alpha}(A) A^{i  - (k+s)}B\big) \big(CF_{\alpha}(A) A^{i  - (k+s)}B\big)^T\big)\]
Using $tr(\phi^T \phi) = \|\phi\|_F^2$ for any matrix $\phi$ we obtain 

\[=  \sum_{i = k+s}^{k+s + L}\|CF_{\alpha}(A) A^{i  - (k+s)}B\|_F^2 = G_{\alpha,L}\]
as desired.  
\fi

Further, we observe that for any fixed value of $W_t$, the overall norm is small, i.e.  $\norm{\wh{y}_{t+k}} \leq 2P_1 \log (1/\eps)$, only if $V_t$ lands in some ball of radius $4P_1 \log (1/\eps)$ around the fixed value of $W_t$. We bound the probability of such an event by invoking anti-concentration properties we can derive from $u_t$ satisfying $(4,2,K)$-hypercontractivity. 
In particular, let $v \in \R^m$ be a unit vector such that 
\[
v^T \Sigma_{V_t} v \geq G_{\alpha,L}/m \geq (100P_1 \log (1/\eps))^2 \,.
\]
Observe, such a direction clearly exists since  $\Sigma_V$ is an $m \times m$ PSD matrix with trace $G_{\alpha, L}$.  Now by Claim~\ref{claim:sum-hypercontractive}, the random variable $V_t$ is $(4,2,K)$-hypercontractive.  Thus, we can apply Claim~\ref{claim:anti-concentration} to the random variable $\la v, V_t \ra/(100 P_1 \log (1/\eps))$.  In order for $V_t$ to land in a fixed ball of radius $4P_1 \log (1/\eps)$, its projection on direction $v$ must land in some fixed interval of width $8 P_1 \log(1/\eps)$.  However, by Claim~\ref{claim:anti-concentration} applied to  $\la v, V_t \ra/(100 P_1 \log (1/\eps))$, this probability is at most $1 - 1/(10K)$. 
% \textcolor{red}{A: note to self, there is a lot going on here, add more exposition.} %\noindent\textcolor{MidnightBlue}{Ankur: We should probably be more explicit about the argument, and the role of $v$ here. We're fixing the realization of $W$ then saying that $V$ has to land in a ball, but if you project onto the direction of $v$ the chance that your projection lands in the projection of the ball is small.} 

Furthermore, the above reasoning about $u_{t - s - L}, \dots , u_{t - s}$ holds for any choices of $u_1, \dots , u_{t-s - L - 1}$ and $w_1, \dots , w_{t-s}$.  Thus, we can multiply the conditional probabilities over different choices of $t$, namely $L, 2L, \dots , 10^2snm^2 K L\log L$.  We get that  
\begin{equation}
\begin{split}
&\Pr\big[\norm{\hat{y}_{iL + k}} \leq 2P_1\log(1/\eps) \quad \forall i \in [10^2snm^2 K \log L]\big]  \\ &= \prod_{i = 1}^{10^2snm^2 K \log L}\Pr\left[\norm{\hat{y}_{iL + k}} \leq 2P_1\log(1/\eps) \Big\vert \norm{\hat{y}_{jL + k}} \leq 2P_1\log(1/\eps) \forall j <  i\right] \\ &\leq \left(1 - \frac{1}{10K}\right)^{10^2s n m^2K \log L} \leq \frac{1}{L^{10s nm^2}}
\end{split}
\end{equation}
where in the above we used that 
\[
\Pr\left[\norm{\hat{y}_{iL + k}} \leq 2P_1\log(1/\eps) \Big\vert \norm{\hat{y}_{jL + k}} \leq 2P_1\log(1/\eps) \forall j <  i\right] \leq 1 - \frac{1}{10K}
\]
because the events being conditioned on depend only on earlier $u_t, w_t, z_t$ and we showed that $[\norm{\hat{y}_{iL + k}} \leq 2P_1\log(1/\eps)$ with at most $1 - \frac{1}{10K}$ probability regardless of what these earlier realizations are. Thus, with $1 - (1/L)^{10s n m^2}$ probability, the given $\alpha$ will actually violate some constraint and not be a feasible solution. 

Note that we take $i \in [10^2 snm^2K \log L]$ where the dominant term is the logarithmic dependence on $L$.  We could not take for example a polynomial dependence on $L$ as this would incur a polynomial dependence in $L$ on bounds on $\Norm{ u_{iL+c}} , \Norm{ w_{iL+c}}, \Norm{ z_{iL+c}}$ in \pref{claim:basic-boundedness}.  This would be unacceptable for the concentration of our estimates of the markov parameters in \pref{lem:variance-bound-v2}.   

\end{proof}

Next, we relate the potential function $G_{\alpha ,L}$ to the variance of our actual estimators for the Markov parameters.  First, we will need to define a similar-looking potential and relate it to $G_{\alpha , L}$. 

\begin{definition}
For an integer $l$ and matrix-valued coefficients $\alpha = (\alpha_1, \dots , \alpha_s)$, define the function 
\[
H_{\alpha , l} = \sum_{i = 0}^l \norm{F_{\alpha}(A)A^i}_F^2 \,.
\]
\end{definition}

Intuitively, this potential simply does not include the matrix $B$ and thus is polynomially related to $G_{\alpha, l}$, since the condition number of the controllability matrix is bounded. We make this precise as follows:

\begin{lemma}[Potential without $B$]\label{lem:alternate-potential}
For a fixed sequences of matrices $\alpha = \Paren{ \alpha_1, \alpha_2, \ldots, \alpha_s}$, we have
\[
H_{\alpha, l} \leq \kappa^2 s  G_{\alpha , l + s}.
\]
\end{lemma}
\begin{proof}

Recall the controllability matrix $Q_s$ in Definition~\ref{def:controllability}.  Since the maximum singular value of $Q_{2s}$ is at least $1$ (since $\norm{B} \geq 1$), the minimum singular value of $Q_s$ must be at least $1/\kappa$.  Thus, we must have 
\[
\norm{F_{\alpha}(A)A^i}_F^2  \leq \kappa^2 \norm{F_{\alpha}(A)A^iQ^\top_s}_F^2 = \kappa^2 \sum_{j = 0}^{s-1}  \norm{F_{\alpha}(A)A^{i + j}B}_F^2 \,.
\]
\noindent
Summing the above over $l$ then gives the desired inequality.
\end{proof}

Next, for any fixed sequence of $\alpha$'s we show that the variance of our estimator can be bounded in terms of the norm of the $\alpha_i$'s and the two potentials we defined above. In particular, 

\begin{lemma}[Potential to Variance Bound]\label{lem:variance-bound-v2}
For any fixed $\alpha_1, \dots , \alpha_s \in \R^{m \times m}$ and any $0 \leq j \leq k$, we have
\[
\EE\left[  \Norm{ \frac{1}{L}\sum_{t = s + 1}^{s + L }  \hat{y}_{t + k}u_{t + k - j}^\top - X_j}_F^2 \right] \leq \frac{P_1}{L}\left( 1 + \norm{\alpha_1}_F^2 + \dots + \norm{\alpha_s}_F^2 + G_{\alpha,L} + H_{\alpha , L}\right) \,.
\]
where recall $X_j$ is defined as in Definition~\ref{def:estimator-means} and the expectation is over the draws of $u_t, w_t, z_t$.
\end{lemma}

\begin{proof}
We can express $\hat{y}_{t + k}$ using Fact~\ref{fact:formula}
\begin{equation}\label{eq:formula}
\begin{split}
    \wh{y}_{t + k} = &   \left( z_{t + k}  - \sum_{ i = 1}^s \alpha_i z_{t-i} \right) + \left( D u_{t+k}   - \sum_{ i = 1 }^s \alpha_i Du_{t-i} \right)  \\ & + \sum_{i = 1}^{k+s}  F_{\alpha}^{(i-1)}(A)\Paren{ Bu_{t+k- i} + w_{t+k  - i} } + \sum_{i = k+s + 1}^{t + k } F_{\alpha}(A) A^{i  - (k+s + 1)}\Paren{ Bu_{t+k- i} + w_{t+k  - i}}
\end{split}
\end{equation}
Also recall that 
\[
\EE \left[ \hat{y}_{t + k} u_{t+k - j}^\top \right]  = X_j\,.
\]
We then use independence of the $u_i,w_i,z_i$ to compute the variance.  Observe,

\begin{equation*}
\begin{split}
    \frac{1}{L}\sum_{t = s+1}^{s + L} \hat{y}_{t + k} u_{t + k - j}^\top  =  \frac{1}{L}\sum_{t = s+1}^{s + L} &  \left( z_{t + k} u_{t+k-j}^\top  - \sum_{ i = 1}^s \alpha_i z_{t-i} u_{t+k-j}^\top  \right) + \left( D u_{t+k} u_{t+k-j}^\top    - \sum_{ i = 1 }^s \alpha_i Du_{t-i}u_{t+k-j}^\top  \right) \\
    & +  \sum_{i = 1}^{k+s}  F_{\alpha}^{(i-1)}(A) \Paren{ Bu_{t+k- i} u_{t+k-j}^\top  + w_{t+k  - i} u_{t+k-j}^\top } \\
    & + \sum_{i = k+s + 1}^{t + k } F_{\alpha}(A) A^{i  - (k+s + 1)}\Paren{ Bu_{t+k- i} u_{t+k-j}^\top + w_{t+k  - i}  u_{t+k-j}^\top}
\end{split}
\end{equation*}
 and each term
is a quadratic of the form $M u_iu_j^\top , M w_iu_j^\top ,  M z_iu_j^\top  $ for some matrix $M$.  Note that the expectations of these terms are $0$ except $Mu_iu_i^\top$ which has expectation $M$.  Thus, when we compute the variance 
\begin{equation}
\label{eqn:var-comp}
\begin{split}
\EE&\left[  \Norm{ \frac{1}{L}\sum_{t = s + 1}^{s + L }  \hat{y}_{t + k}u_{t + k - j}^\top - X_j}_F^2 \right] \\
& = \expecf{}{\textsf{Tr}\Paren{ \Paren{ \frac{1}{L}\sum_{t = s + 1}^{s + L }  \hat{y}_{t + k}u_{t + k - j}^\top - X_j }^\top \Paren{ \frac{1}{L}\sum_{t = s + 1}^{s + L }  \hat{y}_{t + k}u_{t + k - j}^\top - X_j }  } } \\
& = \expecf{}{ \underbrace{ \textsf{Tr}\Paren{ \frac{1}{L^2} \Paren{ \sum_{t = s+1}^{s+L} u_{t+k+j} \hat{y}_{t+k}^\top   } \cdot \Paren{ \sum_{t = s+1}^{s+L}  \hat{y}_{t+k} u_{t+k+j}^\top   }   } }_{\eqref{eqn:var-comp}.(1)} } - \textsf{Tr}\Paren{ X_j^\top X_j  }
\end{split}
\end{equation}
We observe that  the expression \eqref{eqn:var-comp}.(1) has terms of the following types 
\begin{align*}
\textsf{Tr}\left(M_1 (u_{i_1}u_{j_1}^\top -  1_{i_1 = j_1} \cdot I) (u_{j_2}u_{i_2}^\top -   1_{i_2 = j_2} \cdot I) M_2^\top \right), \textsf{Tr}\left(M_1 w_{i_1}u_{j_1}^\top u_{j_2}w_{i_2}^\top M_2^\top \right),  \textsf{Tr}\left(M_1 z_{i_1}u_{j_1}^\top u_{j_2}z_{i_2}^\top M_2^\top \right),  \\ \textsf{Tr}\left(M_1 (u_{i_1}u_{j_1}^\top -  1_{i_1 = j_1} \cdot I)  u_{j_2}w_{i_2}^\top M_2^\top \right), \textsf{Tr}\left(M_1 (u_{i_1}u_{j_1}^\top -  1_{i_1 = j_1} \cdot I)  u_{j_2}z_{i_2}^\top M_2^\top  \right), \textsf{Tr}\left(M_1 w_{i_1}u_{j_1}^\top u_{j_2}z_{i_2}^\top M_2^\top \right)
\end{align*}
and their transposes.   The expectations of each of these terms is $0$ except when $i_1 = i_2$ and $j_1 = j_2$ or terms of the form
\[
\textsf{Tr}\left(M_1 u_{i}u_{j}^\top u_{i}u_{j}^\top M_2^\top \right)
\]
for $i \neq j$.  Each term $ u_iu_j^\top  , w_iu_j^\top ,  z_iu_j^\top  $ appears at most twice since the choice of $j$ uniquely determines the index $t$ in the sum $\sum_{t = s+1}^{s + L} \hat{y}_{t + k} u_{t + k - j}^\top$ and then in the expression for $\hat{y}_{t+k}$ in \eqref{eq:formula}, each individual variable appears at most twice.  Also, note that 
\[
\textsf{Tr}\left(M_1 u_{i}u_{j}^\top u_{i}u_{j}^\top M_2^\top \right) \leq \frac{1}{2} \textsf{Tr}\left(M_1 u_{i}u_{j}^\top u_{j}u_{i}^\top M_1^\top  + M_2 u_{j}u_{i}^\top u_{i}u_{j}^\top M_2^\top \right)
\]
so using the above inequality, we can eliminate all cross terms and only consider terms where $i_1 = j_1$ and $i_2 = j_2$.  Overall, we have
\begin{equation}
\label{eqn:consice-var-bound}
\begin{split}
\EE & \left[  \Norm{ \frac{1}{L}\sum_{t = s + 1}^{s + L }  \hat{y}_{t + k}u_{t + k - j}^\top - X_j}_F^2   \right]  \\
& \leq \underbrace{ \frac{4}{L^2}\sum_{t = s+1}^{s + L}\EE\left[\Norm{z_{t+k}u_{t + k - j}^\top}_F^2 + \sum_{i = 1}^s  \Norm{\alpha_i z_{t- i} u_{t + k - j}^\top}_F^2 \right]  }_{\eqref{eqn:consice-var-bound}.(1)}  \\ & \quad + \underbrace{ \frac{4}{L^2}\sum_{t = s+1}^{s + L}\EE\left[\Norm{D(u_{t+k} u_{t + k - j}^\top - 1_{j = 0} \cdot I)}_F^2 + \sum_{i = 1}^s  \Norm{D\alpha_i u_{t- i}u_{t + k - j}^\top}_F^2 \right] }_{\eqref{eqn:consice-var-bound}.(2)} \\ & \quad + \underbrace{  \frac{4}{L^2}\sum_{t = s+1}^{s + L}\sum_{i = 1}^{k + s} \EE\left[\Norm{F_{\alpha}^{(i-1)}(A)B(u_{t+k- i} u_{t + k - j}^\top - 1_{i = j} \cdot I )}_F^2 \right] }_{\eqref{eqn:consice-var-bound}.(3)}  \\ 
&\quad + \underbrace{ \frac{4}{L^2}\sum_{t = s+1}^{s + L}\sum_{i = 1}^{k + s }  \EE\left[ \Norm{F_{\alpha}^{(i- 1)}(A)w_{t+k - i} u_{t + k - j}^\top }_F^2\right] }_{\eqref{eqn:consice-var-bound}.(4)} \\ 
& \quad + \underbrace{ \frac{4}{L^2}\sum_{t = s+1}^{s + L}\sum_{i = k + s + 1}^{t + k}  \EE\left[\Norm{ F_{\alpha}(A)A^{i - (k + s + 1)}Bu_{t+k - i} u_{t + k - j}^{\top} }_F^2\right] }_{\eqref{eqn:consice-var-bound}.(5)} \\ & 
\quad + \underbrace{ \frac{4}{L^2}\sum_{t = s+1}^{s + L}\sum_{i = k + s + 1}^{t + k}  \EE\left[ \Norm{ F_{\alpha}(A)A^{i - (k + s + 1)} w_{t + k - i} u_{t + k - j}^{\top} }_F^2 \right] }_{\eqref{eqn:consice-var-bound}.(6)} \,.
\end{split}
\end{equation}

We bound each of the terms above as follows: first we observe that as long as $P_1$ is a sufficiently large polynomial in the system parameters, $\calS$, it follows from Claim~\ref{claim:boundA} that 
\begin{equation*}
    \eqref{eqn:consice-var-bound}.(1) + \eqref{eqn:consice-var-bound}.(2) + \eqref{eqn:consice-var-bound}.(3) + \eqref{eqn:consice-var-bound}.(4) \leq  \frac{P_1 }{L} \Paren{ 1 + \norm{\alpha_1}_F^2 + \dots + \norm{\alpha_s}_F^2}  
\end{equation*}

% The sum of the first four terms above are all at most $(P_1/L)(1 + \norm{\alpha_1}_F^2 + \dots + \norm{\alpha_s}_F^2) $ as long as $P_1$ is chosen sufficiently large since using Claim~\ref{claim:boundA}, the summands are bounded by fixed polynomials in the system parameters.  
Next, we use the potentials, $G_{\alpha, L}$ and $H_{\alpha, L}$ to bound the last two terms as follows: 
To bound the last two terms, we have
\begin{equation}
\begin{split}
&\sum_{t = s+1}^{s + L}\sum_{i = k + s + 1}^{t + k}  \EE\left[\norm{ F_{\alpha}(A)A^{i - (k + s + 1)}Bu_{t+k - i} u_{t + k - j}^{\top} }_F^2\right]
\\ &= \sum_{t = s+1}^{s + L}\sum_{i = k + s + 1}^{t + k} \EE\left[ \textsf{Tr} \left( F_{\alpha}(A)A^{i - (k + s + 1)}Bu_{t+k - i} u_{t + k - j}^{\top}  u_{t + k - j} u_{t+k - i}^\top \left( F_{\alpha}(A)A^{i - (k + s + 1)}B \right)^\top \right) \right] 
\\ &= \sum_{t = s+1}^{s + L}\sum_{i = k + s + 1}^{t + k} p \norm{F_{\alpha}(A)A^{i - (k + s + 1)}B  }_F^2  \\ &\leq L p  G_{\alpha,L} \,.
\end{split}
\end{equation}
and similarly
\begin{equation}
\begin{split}
\sum_{t = s+1}^{s + L}\sum_{i = k + s + 1}^{t + k} \EE\left[ \norm{ F_{\alpha}(A)A^{i - (k + s + 1)} w_{t + k - i}u_{t + k - j}^{\top} }_F^2 \right] \leq L p H_{\alpha, L} \,.
\end{split}
\end{equation}
Putting everything together, we conclude that 
\begin{equation*}
\EE\left[  \Norm{ \frac{1}{L}\sum_{t = s + 1}^{s + L }  \hat{y}_{t + k}u_{t + k - j}^\top - X_j}_F^2 \right] \leq \frac{P_1}{L}\left( 1 + \norm{\alpha_1}_F^2 + \dots + \norm{\alpha_s}_F^2 + G_{\alpha,L} + H_{\alpha , L}\right) \,.
\end{equation*}
as desired.
\end{proof}

Unfortunately, we cannot use Lemma~\ref{lem:variance-bound-v2} directly because the choice of $\alpha = (\alpha_1, \dots , \alpha_s)$ that we compute using the program $\calC_{\alpha}$ already depends on the realizations of the $u_t, w_t, z_t$ meaning that there is no more fresh randomness.  
In order to use the randomness over $u_t$, $w_t$ and $z_t$ to bound the variance of our estimator, we need decouple the $\alpha$'s out of the variance expression.

To circumvent this, we will derive a symbolic inequality  from Lemma~\ref{lem:variance-bound-v2} that holds simultaneously for all choices of $\alpha$ and thus can be applied even if $\alpha$ depends on the realizations of $u_t, w_t, z_t$.

Note that $G_{\alpha ,l}$ and $H_{\alpha, l}$ are both quadratic expressions in the $\alpha_i$.  It will be useful to extract out the matrix of coefficients, which we do in the following definition.
\begin{definition}[Coefficient Matrix of a quadratic polynomial]
Define $v_{\alpha} = (1, \alpha_1, \dots , \alpha_s)$ where we view the $\alpha_i$ as formal variables and flatten each of the matrices $\alpha_i$ into a vector and concatenate them so that $v_{\alpha}$ has length $sm^2 + 1$.  Let $G_L$ (respectively $H_L$) be the unique symmetric $(sm^2 + 1) \times (sm^2 + 1)$ matrix such that 
\[
v_{\alpha}^T G_L v_{\alpha} = G_{\alpha , L} \,.
\]
\end{definition}
Note that matrices defined above are unique because we force them to be symmetric so the coefficients of the monomials in the $\alpha_i$ uniquely determine the entries of the matrices.  Also, the entries of $G_l, H_l$ are purely functions of the system parameters $A,B,C,D$.  Now we can prove a symbolic version of Lemma~\ref{lem:variance-bound-v2}.
Intuitively, this symbolic version is a way to decouple the vector valued random variables $u_t, w_t$ and $z_t$ from the matrix random variables $\alpha$, and then use properties of the input distribution on the expressions that are independent of $\alpha$.

\begin{corollary}[Symbolic Matrix Inequality]\label{coro:estimator-variance-symbolic}
Consider the vector of formal variables $v_\alpha = (1, \alpha_1, \dots , \alpha_s)$.  For an integer $j$ with  $0 \leq j \leq k$, define $M_j$ to be the unique symmetric matrix such that
\[
 \Norm{ \frac{1}{L}\sum_{t = s + 1}^{s + L}  \hat{y}_{t + k}u_{t + k - j}^\top - X_j}_F^2   = v_{\alpha}^\top M_j v_{\alpha} ,
\]
where recall $X_j$ is defined in Definition~\ref{def:estimator-means}.  Then we have
\[
\EE[M_j] \preceq \frac{P_1}{L}(I + G_{L} + H_L) ,
\]
where the expectation is over the randomness of the realizations of the $u_i, w_i, z_i$.
\end{corollary}
\begin{proof}
Note that Lemma~\ref{lem:variance-bound-v2} holds over all choices of $\alpha = (\alpha_1, \dots , \alpha_s)$.  Thus, we have for any choice of $\alpha$,
\begin{equation*}
    \begin{split}
 v_{\alpha}^\top \EE[M_j] v_{\alpha} & = \EE\left[ \Norm{ \frac{1}{L}\sum_{t = s + 1}^{s + L}  \hat{y}_{t + k}u_{t + k - j}^\top - X_j}_F^2 \right]\\
 & \leq \frac{P_1}{L}\left( 1 + \norm{\alpha_1}_F^2 + \dots + \norm{\alpha_s}_F^2 + G_{\alpha,L} + H_{\alpha , L}\right)  \\ & \leq \frac{P_1}{L} v_{\alpha}^\top(I + G_{L} + H_L) v_{\alpha} \,.
\end{split}
\end{equation*}
Thus, we must actually have
\[
\EE[M_j] \preceq \frac{P_1}{L}(I + G_{L} + H_L) \,.
\]
\end{proof}

Now we can complete the analysis of our algorithm for learning the Markov parameters.

\begin{proof}[Proof of Theorem \ref{thm:learning-markov}]
 
With probability at least $1 - 0.1\delta$, the event in Claim~\ref{claim:basic-boundedness} holds and we condition on it.  Now we solve the program in Definition~\ref{def:program-v2}.  By Lemma~\ref{lem:feasibility}, it is feasible. Let $\tilde{\alpha} = (\tilde{\alpha}_1, \dots , \tilde{\alpha}_s)$ be a feasible solution.  

Now construct a $\gamma$-net, denoted by $\calT$, over the matrices $(\alpha_1, \dots, \alpha_s)$ in Frobenius norm with $\gamma = 1/L^{4n}$.  Then, observe for any $(\alpha_1, \dots , \alpha_s)$ with $\norm{\alpha_i} \leq P_0$, there exists $(\alpha_1', \dots , \alpha_s') \in \calT$ such that
\[
\sqrt{\norm{\alpha_1 - \alpha_1'}_F^2 + \dots + \norm{\alpha_s - \alpha_s'}_F^2} \leq \frac{1}{L^{4n}} \,.
\]
It is clear that such a net exists with $|\calT| \leq L^{4nsm^2}$.  Now, we can union bound over all the events such that with probability at least $1- 0.2\delta$, simultaneously, for all $\alpha$, for all $i\in [100snm^2 K\log(L)]$, it follows from 
Lemma~\ref{lem:potential-bound-v2} that $\norm{\hat{y}_{i L +k} }^2\leq 2P1 \log(1/\epsilon)$.

Next, we
show that the solution $\hat{\alpha}$ we obtain must satisfy 
\begin{equation}\label{eq:solution-potential}
G_{\hat{\alpha},L} \leq m(200 P_1 \log (1/\eps))^2 \,.
\end{equation}
To see this, assume for the sake of contradiction that the above doesn't hold.  Now round our solution $\hat{\alpha}$ to the nearest $\alpha' = (\alpha'_1, \dots , \alpha'_s)$ in the net.  By Lemma~\ref{claim:powering-bound}, as long as $P_2$ (and recall $L = P_2 \log^2 (1/\eps)/\eps^2$) is chosen sufficiently large, we have 
\[
G_{\alpha', L} \geq m(100 P_1 \log (1/\eps))^2
\]
However,  Lemma~\ref{lem:potential-bound-v2} implies that there is some integer $1 \leq i \leq 10^2snm^2K \log L$ such that 
\[
\norm{y_{iL+k} - \alpha_1' y_{iL-1} - \dots - \alpha_s' y_{iL-s}} \geq 2P_1 \log(1/\eps) \,.
\]
However, using the assumptions in Lemma~\ref{claim:basic-boundedness}, the formula for $y_t$ in Fact~\ref{fact:formula} and the bounds in Lemma~\ref{claim:powering-bound}, and the properties of the net, the above implies that
\[
\norm{y_{iL+k} - \hat{\alpha}_1 y_{iL-1} - \dots - \hat{\alpha}_s y_{iL-s}} \geq P_1 \log (1/\eps)
\]
which contradicts the fact that $(\hat{\alpha}_1, \dots , \hat{\alpha}_s)$ is a feasible solution.  Thus, we actually must have \eqref{eq:solution-potential}.  Now by Lemma~\ref{lem:alternate-potential}, we have $H_{\alpha, L - s} \leq \kappa^2 s m(200 P_1 \log(1/\eps))^2$.  Now let $L' = L - s$ and set
\[
\hat{X}_j = \frac{1}{L'}\sum_{t = s+1}^{s + L'}  \hat{y}_{t+k} u_{t + k - j}^\top 
\]
for all $0 \leq j \leq k$.  Let $M_j$ be defined as in Corollary~\ref{coro:estimator-variance-symbolic} (with $L$ replaced by $L'$).  Let 
\[
\widetilde{M_j} = \frac{P_1}{L'}(I + G_{L'} + H_{L'}) \,.
\]
Note that $\widetilde{M_j}$ is clearly PSD.  Also $M_j$ is always PSD regardless of the realizations of the $u_i, w_i, z_i$, so by Markov's inequality and Corollary~\ref{coro:estimator-variance-symbolic}, we have that with $1 - 0.1\delta/k$ probability
\[
\textsf{Tr}\left(\widetilde{M_j}^{-1/2} M_j \widetilde{M_j}^{-1/2} \right) \leq \frac{10k}{\delta} \EE\left[ \textsf{Tr}\left(\widetilde{M_j}^{-1/2} M_j \widetilde{M_j}^{-1/2} \right) \right] \leq  \frac{10k}{\delta}(sm^2 + 1) \,.
\]
This means that with $1 - 0.1\delta/k$ probability, 
\[
\widetilde{M_j}^{-1/2} M_j \widetilde{M_j}^{-1/2} \preceq \frac{10k(sm^2 +1)}{\delta}
\]
which then implies 
\[
M_j  \preceq \frac{10k(sm^2 +1)}{\delta}\widetilde{M_j} \,.
\]
Assuming that this happens, we have that 
\begin{equation}
\begin{split}
\norm{\hat{X}_j - X_j}_F^2 = v_{\alpha}^\top M_j v_{\alpha} & \leq \frac{10k(sm^2 +1)}{\delta} \frac{P_1}{L'}(1 + \norm{\alpha_1}_F^2 + \dots + \norm{\alpha_s}_F^2 + G_{\alpha,L'} + H_{\alpha, L'})  \\ 
&\leq \eps^2
\end{split}
\end{equation}
where the last inequality uses that $L = P_2 (\log (1/\eps))^2/\eps^2 $ and $P_2$ is chosen sufficiently large.  Finally, union bounding the above over all choices of $j = 0,1, \dots k $ completes the proof for the guarantees of the estimator.

Finally, note that our algorithm runs in polynomial time in all the parameters because the convex program $\calC_\alpha$ admits an efficient separation oracle.  To see this, note that there are only polynomially many constraints in $\calC_\alpha$ and each one is either a linear constraint or an ellipsoid constraint both of which admit an efficient separation oracle.
% \abnote{ add a couple of sentences about the running time, should be poly in every thing.}

\end{proof}

\section{From Markov Parameters to System Parameters}
\label{sec:markov-params-to-sys-params}
Note that Theorem~\ref{thm:learning-markov} guarantees that we can get good estimates for the markov parameters.  To complete the proof of our full learning result, Theorem~\ref{thm:main-learning-lds}, we apply the Ho-Kalman algorithm black box to extract the system matrices $\{A,B,C,D\}$ where $\{A,B,C\}$ are recovered up to a similarity transformation.  Recall that linear dynamical systems are specified only up to similarity transformation (see \cite{oymak2019non} for a discussion on this point).  

%\textcolor{red}{A: For completeness, we should add the robust Ho-Kalman algorithm by \cite{oymak2019non}.}

\begin{mdframed}
  \begin{algorithm}[Robust Ho-Kalman, Algortihm 1 in ~\cite{oymak2019non} ]
    \label{algo:robust-ho-kalman}\mbox{}
    \begin{description}
    \item[Input:] Parameter $s$, Markov parameter matrix estimate $\hat{G} = [\hat{X}_0, \dots , \hat{X}_{2s}] $  
    
    \item[Operation:]
    %\mbox{State-space realization $\hat{A}, \hat{B}, \hat{C}$}
    \begin{enumerate}
    \item Set $\hat{D} = \hat{X}_0$
    \item Form the Hankel matrix $\hat{H} \in \R^{ms \times p(s+1)}$ from $\hat{G}$ as
    \[
    \hat{H} = \begin{bmatrix} \hat{X}_1 & \hat{X}_2 & \dots & \hat{X}_{s+1} \\ \hat{X}_2 & \hat{X}_3 & \dots & \hat{X}_{s + 2} \\ \vdots & \vdots & \ddots & \vdots \\ \hat{X}_{s} & \hat{X}_{s+1} & \dots & \hat{X}_{2s} 
    \end{bmatrix}
    \]
    \item $\hat{H}^{-} \in \R^{ms \times ps} \leftarrow $ first $ps$ columns of $\hat{H}$
    \item $\hat{L} \in \R^{ms \times ps} \leftarrow$ rank $n$ approximation of $\hat{H}^{-}$ obtained via SVD
    \item $U,\Sigma,V = SVD(\hat{L})$
    \item $\hat{O} \in \R^{ms \times n} \leftarrow U\Sigma^{1/2}$
    \item $\hat{Q} \in \R^{n \times ps} \leftarrow \Sigma^{1/2}V^*$
    \item $\hat{C} \leftarrow$ first $m$ rows of $\hat{O}$
    \item $\hat{B} \leftarrow $ first $p$ column of $\hat{Q}$
    \item $\hat{H}^{+} \in \R^{ms \times ps} \leftarrow $ last $ps$ column of $\hat{H}$
    \item $\hat{A} \leftarrow \hat{O}^{\top} \hat{H}^{+} \hat{Q}^{\top}$
    \end{enumerate}
    \item[Output:] $\hat{A} \in \R^{n \times n}, \hat{B} \in \R^{n \times p}, \hat{C} \in \R^{m \times n}, \hat{D} \in \R^{m \times p}$  
    \end{description}
  \end{algorithm}
\end{mdframed}

% \textcolor{red}{A: can we add the def of SVD, $\top$, $+$ and $A^{1/2}$ to the prelims?}

%We aim to estimate the following ground truth matrix $G$ of markov parameters defined to be 

The main point of this section is to show that if the input $\hat{G}$ to Algorithm~\ref{algo:robust-ho-kalman} is close to the true Markov parameters
\[G = [D, CB, CAB, ..., CA^{2s-1}B] \in \R^{m \times (2s+1)p}\]
then the actual estimates of $\hat{A}, \hat{B},\hat{C},\hat{D}$ output by the algorithm must be close to the true parameters up to a common rotation.  Once we have this, then combining with Theorem~\ref{thm:learning-markov} will complete the proof of our main theorem, Theorem~\ref{thm:main-learning-lds}.

%Using our estimate of the Markov parameters from \pref{lem:markov-param} we form $\hat{G}$ with the guarantee 
The following lemma from \cite{oymak2019non} establishes error guarantees for the Ho-Kalman algorithm given operator norm bounds on estimating $G$.   
\begin{lemma}[\cite{oymak2019non}]\label{lem:ho-kalman}
For observability and controllability matrices that are rank $n$, the Ho-Kalman algorithm applied to $\hat{G}$ produces estimates $\hat{A},\hat{B}$, and $\hat{C}$ such that 
there exists similarity transform $T \in \R^{n \times n}$ such that 

\[\max\{\|C - \hat{C} T\|_F, \| B - T^{-1} \hat{B}\|_F\} \leq 5\sqrt{n \| G - \hat{G}\|}
\]
 and 
 \[\|A - T^{-1} \hat{A}T\|_F \leq \frac{\sqrt{n \|G - \hat{G}\|} \| H\| }{\sigma_{min}^{3/2}(H^{-})}\]
 
 and 
 \[\Norm{D - \hat{D}}_F \leq \sqrt{n}\Norm{G - \hat{G}}\]
 where in the above
 \[
 G = [D, CB, CAB, ..., CA^{2s-1}B]  \,.
 \]
\end{lemma}

A straightforward application of this lemma allows us to complete the proof of Theorem~\ref{thm:main-learning-lds}.   

\iffalse
\begin{corollary}
For trajectory of length $\poly(\mathcal{S})\log^2(\frac{1}{\eps})/\eps^2$ and with probability $1-\delta$, the Ho-Kalman algorithm applied to the Hankel matrix formed from the markov parameters in \pref{thm:learning-markov} outputs estimates $\hat{A}, \hat{B},$ and $\hat{C}$ such that there exists a unitary matrix $T \in \R^{n \times n}$ satisfying 

\[\max\{\|C - \hat{C} T\|_F, \| B - T^* \hat{B}\|_F\} \leq \poly(\mathcal{S}) \cdot \epsilon
\]
and 
\[\|A - T^* \hat{A}T\|_F \leq \poly(\mathcal{S}) \cdot \epsilon \]
and 
\[\Norm{D - \hat{D}} \leq \poly(\mathcal{S}) \cdot \epsilon\]
\end{corollary}
\fi
\begin{proof}[Proof of Theorem~\ref{thm:main-learning-lds}]
By Theorem~\ref{thm:learning-markov}, in Algorithm~\ref{algo:learning-lds}, the input $\hat{G}$ to the Ho-Kalman algorithm satisfies 
\[
\norm{G - \hat{G}}_F \leq \sqrt{2s + 1}\eps
\]
with probability at least $1 - \delta$.  Now we apply Lemma~\ref{lem:ho-kalman}.  The two things we need to do are upper bound $\norm{ H}$ and lower bound $\sigma_{min} (H^{-})$.  We have an upper bound on $\| H\| \leq \sigma_{max}(\mathcal{O}_s)\sigma_{max}(\mathcal{Q}_s) \leq \kappa^2 s \|B\| \|C\|$ where we use \pref{claim:eigenvalue-bound}.  We also have $\sigma_{min}(H^{-}) \geq \sigma_{min}(\mathcal{O}_s) \sigma_{min}(\mathcal{Q}_s) \geq \norm{B}\norm{C} \geq 1$.  Therefore we conclude that there is a similarity transform $T$ such that 
\[
\begin{split}
\max\{\|A - T^{-1} \hat{A}T\|_F, \|B - \hat{B} T\|_F, \| C - T^{-1} \hat{C}\|_F\} \leq \poly(n,\kappa, s, \norm{B}, \norm{C}) \sqrt{\epsilon} \,. 
\end{split}
\]
Redefining $\eps$ appropriately immediately gives the desired result.
\end{proof}

\input{lower-bound}

\bibliographystyle{alpha}
\bibliography{local,scholar}

\appendix
\input{appendix}

\end{document}

%% file: lower-bound.tex
\section{Sample Complexity Lower Bound for Ill-Conditional LDS}
\label{sec:lower-bound}

In this section, we prove a lower bound, that when the observability or controllability matrix of an LDS is close to singular, then it is information-theoretically impossible to learn.  We consider the case where the distributions $\calD_u = \calD_0 =  N(0, I)$ and $\calD_w = N(0, \Sigma_w)$ where $\Sigma_w$ will be set later.  For simplicity, we also set $x_0 = 0$ and also $D = 0$.   
%\begin{remark}
%The lower bound proof relies on the fact that $\calD_u, \calD_w, \calD_z$ are Gaussian but the covariances don't matter as long as they are all polynomially well conditioned.
%\end{remark}

\begin{definition}
We say that an LDS $\calL(A,B,C,D)$ is $(\delta,v)$-unobservable if $v$ is a unit vector such that for all integers $s \geq 0$,
\[
\norm{CA^s v} \leq \delta \,.
\]
Note that the above condition depends only on $A,C$ so we will sometimes talk about a pair of matrices $A,C$ being $(\delta,v)$-unobservable.
\end{definition}

\begin{definition}
We say that an LDS $\calL(A,B,C,D)$ is $(\delta,v)$-uncontrollable if $v$ is a unit vector such that for all integers $s \geq 0$,
\[
\norm{(A^sB)^\top v} \leq \delta \,.
\]
Note that the above condition depends only on $A,B$ so we will sometimes talk about a pair of matrices $A,B$ being $(\delta,v)$-unobservable.
\end{definition}

\begin{definition}
For an LDS $\calL = \calL(A,B,C,D)$ and integer $t \geq 0$, we define the distribution $\calD_{\calL, t}$ to be the joint distribution of $(u_0, \dots , u_t, y_0, \dots , y_t)$.  We define $\Sigma_{\calL , t} \in \C^{(m + p)(t + 1)}$ to be the covariance of this distribution (where we flatten and then concatenate all of the $u_i, y_i$).
\end{definition}

Clearly the joint distribution of $(u_0, \dots , u_T, y_0, \dots , y_T)$ is Gaussian and has mean $0$.  Thus, we have the following fact.
\begin{fact}\label{fact:gaussian-distr}
We have
\[
\calD_{\calL, T} = N(0, \Sigma_{\calL , T}) \,.
\]
\end{fact}

We also need the following formulas.

\begin{fact}\label{fact:obs-cov-formula}
For $t_1 \geq t_2 \geq 0$, we have
\[
\EE[y_{t_1} y_{t_2}^\top ] = \sum_{i = 1}^{t_2} \left(CA^{t_1 - t_2 + i - 1}B (CA^{ i - 1}B )^\top + (CA^{t_1 - t_2 + i - 1})\Sigma_w (CA^{t_1 - t_2 + i - 1})^\top\right) +   1_{t_1 = t_2} (DD^\top + I) \,. 
\]
\end{fact}
\begin{proof}
Recall the formula from Fact~\ref{fact:formula}:
\begin{equation*}
    y_t = \sum_{i = 1}^{t} \Paren{  C A^{i-1} B   u_{t-i}  + C A^{i-1}w_{t-i} } +  Du_t + z_{t}  \,.
\end{equation*}
Now using the setting of $\calD_u, \calD_w, \calD_z$ and the independence of $u_i,w_i,z_i$, we get the desired relation.
\end{proof}

\begin{definition}
For matrices $A,B,C$ and parameter $T$, let
\[
P_T(A,B,C) = \begin{bmatrix}
0 & 0 & \dots & 0 & 0\\ CB & 0 & \dots &  0 & 0 \\ \vdots & \vdots & \ddots & \vdots & \vdots \\ CA^{T-1}B & CA^{T-2}B & \dots  & CB &  0 \,.
\end{bmatrix} \,.
\]
where there are $T+1$ rows and columns in the block matrix.
\end{definition}

\begin{claim}\label{claim:covariance-formula}
For an LDS $\calL(A,B,C,D )$ with $D = 0$ and noise distribution $\calD_u = \calD_z = N(0,I)$ and $\calD_w = N(0, \Sigma_w)$, we have
\begin{equation*}
\begin{split}
\Sigma_{\calL, T} & = \begin{bmatrix} I_{p(t+1)} \\ P_T(A,B,C) \end{bmatrix} \begin{bmatrix} I_{p(t+1)} & P_T(A,B,C)^\top \end{bmatrix} \\
& \hspace{0.2in}+ \begin{bmatrix} 0 \\ P_T(A,\Sigma_w^{1/2},C) \end{bmatrix} \begin{bmatrix} 0 &  P_T(A,\Sigma_w^{1/2},C)^\top \end{bmatrix}  + \begin{bmatrix}
0 & 0 \\ 0 & I_{m(t+1)}
\end{bmatrix} \,.
\end{split}
\end{equation*}
\end{claim}
\begin{proof}
Recall that $\EE[y_{t_1} u_{t_2}^\top ] $ is $CA^{t_1 - t_2 - 1}$ if $t_1 > t_2$ and is $0$ otherwise.  Combining this with the formula in Fact~\ref{fact:obs-cov-formula} gives the desired relation.
\end{proof}

Now we prove our lower bound assuming that the observability matrix is ill-conditioned.  A similar construction works when the controllability matrix is ill-conditioned.  In particular, Lemma~\ref{lem:observability-lowerbound} says that if $\delta$ is exponentially small, then we need $T$ to be exponentially large to distinguish $\calL$ and $\calL'$ with constant advantage.  

\begin{lemma}\label{lem:observability-lowerbound}
Let $\calL = \calL(A,B,C,D)$ be an LDS that is $(\delta,v)$-unobservable and assume that $D = 0$ and the noise distributions are $\calD_{u} = \calD_z = N(0,I)$ and $\calD_w = N(0,BB^\top)$ \footnote{Any choice of $\Sigma_w$ with $\Sigma_w \succeq BB^\top$ will suffice}.  Let $u \in \R^{p}$ be an arbitrary vector and let  $\calL' = \calL(A  ,  B + vu^\top ,C  ,D)$ be another LDS with the same noise distributions.  Then for $T^2 (m + p) \norm{u} \leq 1/(100 \delta)$,
\[
d_{\textsf{TV}}( \calD_{\calL, T}, \calD_{\calL',T} ) \leq  O(T^2(m + p)\norm{u}\delta) \,.
\]
\end{lemma}

\begin{proof}
Let $\omega \in \C^{(m + p)(t+ 1)}$ be a unit vector and define $\omega = (\omega_u, \omega_y)$ where $\omega_u$ is the first $(t+1)p$ coordinates of $\omega$ and $\omega_y$ is the last $(t+1)m$ coordinates.  We will bound the difference
\[
\omega^\top (\Sigma_{\calL, T} - \Sigma_{\calL', T} )\omega
\]
and since $\omega$ was arbitrary, we will use this to deduce closeness between the covariance matrices which will then imply closeness in statistical distance of the corresponding distributions.  Define the matrices $P = P_T(A,B,C)$ and $P' = P_T(A, B + vu^\top,C)$.  By the assumption on the system,
\[
\norm{P - P'}_F \leq (T+1)\delta \norm{u} \,.
\]
Now, we have
\[
\begin{split}
&\left \lvert \omega^\top (\Sigma_{\calL, T} - \Sigma_{\calL', T} )\omega \right\rvert = \left\lvert 2 \omega_y^\top (P - P') \omega_u + \omega_y^\top (PP^\top - P'(P')^\top) \omega_y \right\rvert \\ &\leq (T+1)\delta \norm{u} \norm{\omega_y} \norm{\omega_u} + \left\lvert \omega_y^\top\left((P - P')P^\top + P(P - P')^\top - (P - P')(P - P')^\top \right)\omega_y \right\rvert  \\ &\leq (T+1)\delta \norm{u} \norm{\omega_y} \norm{\omega_u} + 2(T + 1)\delta \norm{u} \norm{\omega_y} \norm{P^\top\omega_y} + ((T+1)\delta \norm{u} \norm{\omega_y})^2 \,.
\end{split}
\]
Now we will lower bound $\omega^\top \Sigma_{\calL, T} \omega$.  Using the formula in Claim~\ref{claim:covariance-formula}, we have
\[
\omega^\top \Sigma_{\calL, T} \omega \geq \min\left( \norm{P^\top \omega_y}^2, \norm{\omega_y}^2 , (1 - \norm{\omega_y} - \norm{P^\top \omega_y})^2 \right)
\]
and thus we conclude
\[
\left \lvert \omega^\top (\Sigma_{\calL, T} - \Sigma_{\calL', T} )\omega \right\rvert \leq 6(T+1)\delta \norm{u}\omega^\top \Sigma_{\calL, T} \omega 
\]
which is equivalent to saying
\[
(1 - 6(T+1) \norm{u} \delta) \Sigma_{\calL, T} \preceq \Sigma_{\calL', T} \preceq (1 + 6(T+1)\norm{u} \delta) \Sigma_{\calL, T} \,.
\]
Now by Fact~\ref{fact:gaussian-distr} and standard bounds on TV distance between Gaussians, we conclude that 
\[
d_{\textsf{TV}}( \calD_{\calL, T}, \calD_{\calL',T} )  = d_{\textsf{TV}}( N(0, \Sigma_{\calL, T}),  N(0, \Sigma_{\calL', T}) )  \leq O(T^2(m+p)\norm{u} \delta) 
\]
and we are done.
\end{proof}

We have an analogous lower bound when the controllability matrix is ill-conditioned.
\begin{lemma}\label{lem:controllability-lowerbound}
Let $\calL = \calL(A,B,C,D)$ be an LDS that is $(\delta,v)$-uncontrollable and assume that $D = 0$ and the noise distributions are $\calD_{u} = \calD_z = N(0,I)$ and $\calD_w = N(0,BB^\top)$.  Let $u \in \R^{m}$ be an arbitrary vector and let  $\calL' = \calL(A  ,  B ,C + uv^\top  ,D)$ be another LDS with the same noise distributions.  Then for $T^2 (m + p) \norm{u} \leq 1/(100 \delta)$,
\[
d_{\textsf{TV}}( \calD_{\calL, T}, \calD_{\calL',T} ) \leq  O(T^2(m + p) \norm{u} \delta) \,.
\]
\end{lemma}
\begin{proof}
The proof is essentially the same as the proof of Lemma~\ref{lem:observability-lowerbound}.
\end{proof}

Putting together Lemma~\ref{lem:observability-lowerbound} and Lemma~\ref{lem:controllability-lowerbound}, we can prove our full lower bound.  We need a minor assumption that $A$ is generic, in particular, we need that it is not too close to a multiple of the identity plus a rank-$1$ perturbation.  Essentially all matrices satisfy this assumption as long as $n \geq 3$. 

\begin{definition}
We say a matrix $A \in \R^{n \times n}$ and vector $v$ are $c$-generic if $\norm{A} \geq c$ and there are unit vectors $u,w$ such that 
\[
\begin{split}
\la u, v \ra = 0 \\
\la u , w \ra = 0 \\
\la u , Aw \ra \geq c \norm{A} \,.
\end{split}
\]
\end{definition}

\begin{theorem}\label{thm:lower-bound}
Let $A,C$ be matrices that are $(\delta,v)$-unobservable for some $0 < \delta  < 0.1$ and unit vector $v$.  Assume that $(A,v)$ are $c$-generic for some constant $c$.  Then any algorithm that is given an LDS $\calL(A,B,C,D)$ that uses at most 
\[
o\left( \frac{1}{\sqrt{\delta (m + p) } } \right)
\]
samples has probability at least $0.4$ of outputting $\hat{A},\hat{B},\hat{C},\hat{D}$ such that there is no invertible matrix $U$ with
\[
\norm{A - U^{-1}AU}_F, \norm{B - U^{-1} B}_F , \norm{C - CU}_F \leq 0.1c^2 \,.
\]
Similarly, the same holds if $A,B$ are matrices that are $(\delta,v)$-uncontrollable.
\end{theorem}

\begin{proof}
Consider $\calL(A,B,C,D)$ where we set $p = n$ and $B = I_n, D = 0$.  Now choose $u,w$ to be unit vectors such that $\la u , v \ra = 0$ and $\la u,w \ra = 0$, $\la u, Aw \ra \geq c\norm{A}$ which exist by the assumption that $(A,v)$ are $c$-generic.  

Construct the the alternate LDS in Lemma~\ref{lem:observability-lowerbound} with the above setting of $u$.  We claim that the parameters of this alternate LDS are not close to  $\calL(A,B,C,D)$ up to any similarity transformation.  To see this, note that we must have 
\[
\norm{I + vu^\top - U^{-1}}_F \leq 0.1c \,.
\]
Also, the inverse of $I + vu^\top$ is $I - vu^\top$ and these are both well-conditioned.  Thus,
\[
\norm{I - vu^\top - U}_F \leq 0.2c \,.
\]
But now, we must have
\[
\begin{split}
&\norm{U^{-1}AU - A}_F  \\ &\geq -\norm{(I + vu^\top - U^{-1})AU}_F  - \norm{(I + vu^\top)A(I - vu^\top - U)}_F + \norm{vu^\top A + Avu^\top - vu^\top A vu^\top }_F 
\\ &\geq -0.5c \norm{A}  + \norm{vu^\top A w } \geq 0.5 c^2 \,.  
\end{split}
\]
On the other hand, by Lemma~\ref{lem:observability-lowerbound}, no algorithm can distinguish between $\calL(A,B,C,D)$ and $\calL(A,B + vu^\top, C,D)$ with better than $0.01$ advantage given 
\[
o\left( \frac{1}{\sqrt{\delta (m + p) } } \right)
\]
samples so thus with $0.4$ probability, the algorithm outputs a bad estimate of $(A,B,C,D)$.  The argument when the system is not controllable is similar, using Lemma~\ref{lem:controllability-lowerbound}.
\end{proof}

%% file: appendix.tex
\section{Variance Diverges for Unstabilized Estimator}

While $y_{t+j}u_t^\top$ is an unbiased estimator for the Markov parameters, here we show a very simple example where the variance is too large so that the empirical estimate of $\EE[y_{t+j}u_t^\top]$ actually has very bad accuracy no matter how many observations we get.

\begin{lemma}\label{lem:variance-blowup}
Consider an LDS where $m = n = p = 1$, $A = B = C = D = 1$ and $u_t,z_t$ are  drawn from $N(0,1)$ and $w_t$ is drawn from $N(0, 100)$ and $x_0 = 0$.  Then for any time window $T$, 
\[
\EE\left[ \frac{1}{T}\sum_{t = 1}^T y_tu_t - 1 \right] \geq 20 \,.
\]
\end{lemma}
\begin{proof}
Expanding out the recursion in the definition of an LDS, we have
\[
y_t = z_t + u_t +  (u_{t-1} + w_{t-1}) + \dots + (u_0 + w_0) \,.
\]
Define the random variable $Q_T = \frac{1}{T}\sum_{t = 1}^T y_tu_t - 1$.  We can write
\[
Q_T = \frac{1}{T} \left( \sum_{t = 1}^T (u_t^2 - 1) + \sum_{t = 1}^T z_tu_t + \sum_{0 \leq t_1 < t_2 \leq T } (u_{t_1}u_{t_2} + w_{t_1}u_{t_2}) \right) \,.
\]
Now using independence, we have
\[
\EE[Q_T^2] \geq  \frac{1}{T^2}\left(50T(T-1)\right) \geq 20 \,,
\]
and this completes the proof.
\end{proof}

%% file: main.bbl
\newcommand{\etalchar}[1]{$^{#1}$}
\begin{thebibliography}{CKMY22b}

\bibitem[{\AA}E71]{aastrom1971system}
Karl~Johan {\AA}str{\"o}m and Peter Eykhoff.
\newblock System identification—a survey.
\newblock {\em Automatica}, 7(2):123--162, 1971.

\bibitem[Ath74]{athans1974importance}
Michael Athans.
\newblock The importance of kalman filtering methods for economic systems.
\newblock In {\em Annals of Economic and Social Measurement, Volume 3, number
  1}, pages 49--64. NBER, 1974.

\bibitem[BK20a]{bakshi2020list}
Ainesh Bakshi and Pravesh Kothari.
\newblock List-decodable subspace recovery via sum-of-squares.
\newblock {\em arXiv preprint arXiv:2002.05139}, 2020.

\bibitem[BK20b]{bakshi2020outlier}
Ainesh Bakshi and Pravesh Kothari.
\newblock Outlier-robust clustering of non-spherical mixtures.
\newblock {\em arXiv preprint arXiv:2005.02970}, 2020.

\bibitem[BP21]{bakshi2021robust}
Ainesh Bakshi and Adarsh Prasad.
\newblock Robust linear regression: Optimal rates in polynomial time.
\newblock In {\em Proceedings of the 53rd Annual ACM SIGACT Symposium on Theory
  of Computing}, pages 102--115, 2021.

\bibitem[Bre15]{bresler2015efficiently}
Guy Bresler.
\newblock Efficiently learning ising models on arbitrary graphs.
\newblock In {\em Proceedings of the forty-seventh annual ACM symposium on
  Theory of computing}, pages 771--782, 2015.

\bibitem[CAT{\etalchar{+}}20]{cherapanamjeri2020optimal}
Yeshwanth Cherapanamjeri, Efe Aras, Nilesh Tripuraneni, Michael~I Jordan,
  Nicolas Flammarion, and Peter~L Bartlett.
\newblock Optimal robust linear regression in nearly linear time.
\newblock {\em arXiv preprint arXiv:2007.08137}, 2020.

\bibitem[CFG14]{candes2014towards}
Emmanuel~J Cand{\`e}s and Carlos Fernandez-Granda.
\newblock Towards a mathematical theory of super-resolution.
\newblock {\em Communications on pure and applied Mathematics}, 67(6):906--956,
  2014.

\bibitem[CHK{\etalchar{+}}20]{cherapanamjeri2020algorithms}
Yeshwanth Cherapanamjeri, Samuel~B Hopkins, Tarun Kathuria, Prasad Raghavendra,
  and Nilesh Tripuraneni.
\newblock Algorithms for heavy-tailed statistics: Regression, covariance
  estimation, and beyond.
\newblock In {\em Proceedings of the 52nd Annual ACM SIGACT Symposium on Theory
  of Computing}, pages 601--609, 2020.

\bibitem[CKMY22a]{chen2022kalman}
Sitan Chen, Frederic Koehler, Ankur Moitra, and Morris Yau.
\newblock Kalman filtering with adversarial corruptions.
\newblock In {\em Proceedings of the 54th Annual ACM SIGACT Symposium on Theory
  of Computing}, pages 832--845, 2022.

\bibitem[CKMY22b]{chen2022online}
Sitan Chen, Frederic Koehler, Ankur Moitra, and Morris Yau.
\newblock Online and distribution-free robustness: Regression and contextual
  bandits with huber contamination.
\newblock In {\em 2021 IEEE 62nd Annual Symposium on Foundations of Computer
  Science (FOCS)}, pages 684--695. IEEE, 2022.

\bibitem[CMY20]{cherapanamjeri2020list}
Yeshwanth Cherapanamjeri, Sidhanth Mohanty, and Morris Yau.
\newblock List decodable mean estimation in nearly linear time.
\newblock {\em arXiv preprint arXiv:2005.09796}, 2020.

\bibitem[DCWS03]{doretto2003dynamic}
Gianfranco Doretto, Alessandro Chiuso, Ying~Nian Wu, and Stefano Soatto.
\newblock Dynamic textures.
\newblock {\em International Journal of Computer Vision}, 51(2):91--109, 2003.

\bibitem[DHKK20]{diakonikolas2020robustly}
Ilias Diakonikolas, Samuel~B Hopkins, Daniel Kane, and Sushrut Karmalkar.
\newblock Robustly learning any clusterable mixture of gaussians.
\newblock {\em arXiv preprint arXiv:2005.06417}, 2020.

\bibitem[Din13]{ding2013}
Feng Ding.
\newblock Two-stage least squares based iterative estimation algorithm for
  cararma system modeling.
\newblock {\em Applied Mathematical Modelling}, 37(7):4798--4808, 2013.

\bibitem[DM22]{DM2022}
Boualem Djehiche and Othmane Mazhar.
\newblock Efficient learning of hidden state lti state space models of unknown
  order, 2022.

\bibitem[DMM{\etalchar{+}}20]{dean2020sample}
Sarah Dean, Horia Mania, Nikolai Matni, Benjamin Recht, and Stephen Tu.
\newblock On the sample complexity of the linear quadratic regulator.
\newblock {\em Foundations of Computational Mathematics}, 20(4):633--679, 2020.

\bibitem[Don92]{donoho1992superresolution}
David~L Donoho.
\newblock Superresolution via sparsity constraints.
\newblock {\em SIAM journal on mathematical analysis}, 23(5):1309--1331, 1992.

\bibitem[FA75]{fetzer1975observability}
Erwin~Enrique Fetzer and PM~Anderson.
\newblock Observability in the state estimation of power systems.
\newblock {\em IEEE transactions on power Apparatus and Systems},
  94(6):1981--1988, 1975.

\bibitem[Fat20]{salar2020}
Salar Fattahi.
\newblock Learning partially observed linear dynamical systems from logarithmic
  number of samples.
\newblock {\em CoRR}, abs/2010.04015, 2020.

\bibitem[FTM17]{faradonbeh2017}
Mohamad Kazem~Shirani Faradonbeh, Ambuj Tewari, and George Michailidis.
\newblock Finite time identification in unstable linear systems.
\newblock {\em CoRR}, abs/1710.01852, 2017.

\bibitem[FTM18]{faradonbeh2018finite}
Mohamad Kazem~Shirani Faradonbeh, Ambuj Tewari, and George Michailidis.
\newblock Finite time identification in unstable linear systems.
\newblock {\em Automatica}, 96:342--353, 2018.

\bibitem[GA10]{grewal2010applications}
Mohinder~S Grewal and Angus~P Andrews.
\newblock Applications of kalman filtering in aerospace 1960 to the present
  [historical perspectives].
\newblock {\em IEEE Control Systems Magazine}, 30(3):69--78, 2010.

\bibitem[Gal16]{Galrinho2016LeastSM}
Miguel Galrinho.
\newblock Least squares methods for system identification of structured models.
\newblock 2016.

\bibitem[GH96]{ghahramani1996parameter}
Zoubin Ghahramani and Geoffrey~E Hinton.
\newblock Parameter estimation for linear dynamical systems.
\newblock 1996.

\bibitem[GLS{\etalchar{+}}20]{ghai2020}
Udaya Ghai, Holden Lee, Karan Singh, Cyril Zhang, and Yi~Zhang.
\newblock No-regret prediction in marginally stable systems.
\newblock {\em CoRR}, abs/2002.02064, 2020.

\bibitem[HK66]{ho1966effective}
BL~HO and Rudolf~E K{\'a}lm{\'a}n.
\newblock Effective construction of linear state-variable models from
  input/output functions.
\newblock {\em at-Automatisierungstechnik}, 14(1-12):545--548, 1966.

\bibitem[HL18]{hopkins2018mixture}
Samuel~B Hopkins and Jerry Li.
\newblock Mixture models, robustness, and sum of squares proofs.
\newblock In {\em Proceedings of the 50th Annual ACM SIGACT Symposium on Theory
  of Computing}, pages 1021--1034, 2018.

\bibitem[HLS{\etalchar{+}}18]{hazan2018spectral}
Elad Hazan, Holden Lee, Karan Singh, Cyril Zhang, and Yi~Zhang.
\newblock Spectral filtering for general linear dynamical systems.
\newblock {\em Advances in Neural Information Processing Systems}, 31, 2018.

\bibitem[HMR18]{hardt2018gradient}
Moritz Hardt, Tengyu Ma, and Benjamin Recht.
\newblock Gradient descent learns linear dynamical systems.
\newblock {\em Journal of Machine Learning Research}, 19:1--44, 2018.

\bibitem[Hop18]{hopkins2018sub}
Samuel~B Hopkins.
\newblock Sub-gaussian mean estimation in polynomial time.
\newblock {\em arXiv preprint arXiv:1809.07425}, 2018.

\bibitem[HSZ17]{hazan2017}
Elad Hazan, Karan Singh, and Cyril Zhang.
\newblock Learning linear dynamical systems via spectral filtering.
\newblock {\em CoRR}, abs/1711.00946, 2017.

\bibitem[IK22]{ivkov2022list}
Misha Ivkov and Pravesh~K Kothari.
\newblock List-decodable covariance estimation.
\newblock In {\em Proceedings of the 54th Annual ACM SIGACT Symposium on Theory
  of Computing}, pages 1276--1283, 2022.

\bibitem[JLST21]{jambulapati2021robust}
Arun Jambulapati, Jerry Li, Tselil Schramm, and Kevin Tian.
\newblock Robust regression revisited: Acceleration and improved estimation
  rates.
\newblock {\em Advances in Neural Information Processing Systems},
  34:4475--4488, 2021.

\bibitem[Kal60a]{kalman1960general}
Rudolf~E Kalman.
\newblock On the general theory of control systems.
\newblock In {\em Proceedings First International Conference on Automatic
  Control, Moscow, USSR}, pages 481--492, 1960.

\bibitem[Kal60b]{kalman1960new}
Rudolph~Emil Kalman.
\newblock A new approach to linear filtering and prediction problems.
\newblock 1960.

\bibitem[KKK19]{karmalkar2019list}
Sushrut Karmalkar, Adam Klivans, and Pravesh Kothari.
\newblock List-decodable linear regression.
\newblock In {\em Advances in Neural Information Processing Systems}, pages
  7423--7432, 2019.

\bibitem[KKM18]{klivans2018efficient}
Adam Klivans, Pravesh~K Kothari, and Raghu Meka.
\newblock Efficient algorithms for outlier-robust regression.
\newblock {\em arXiv preprint arXiv:1803.03241}, 2018.

\bibitem[KKMM20]{kelner2020learning}
Jonathan Kelner, Frederic Koehler, Raghu Meka, and Ankur Moitra.
\newblock Learning some popular gaussian graphical models without condition
  number bounds.
\newblock {\em Advances in Neural Information Processing Systems},
  33:10986--10998, 2020.

\bibitem[KS17]{kothari2017outlier}
Pravesh~K Kothari and David Steurer.
\newblock Outlier-robust moment-estimation via sum-of-squares.
\newblock {\em arXiv preprint arXiv:1711.11581}, 2017.

\bibitem[KSS18]{kothari2018robust}
Pravesh~K Kothari, Jacob Steinhardt, and David Steurer.
\newblock Robust moment estimation and improved clustering via sum of squares.
\newblock In {\em Proceedings of the 50th Annual ACM SIGACT Symposium on Theory
  of Computing}, pages 1035--1046. ACM, 2018.

\bibitem[LAHA20]{LAHA2020}
Sahin Lale, Kamyar Azizzadenesheli, Babak Hassibi, and Anima Anandkumar.
\newblock Regret bound of adaptive control in linear quadratic gaussian {(LQG)}
  systems.
\newblock {\em CoRR}, abs/2003.05999, 2020.

\bibitem[Lee20]{lee2020}
Holden Lee.
\newblock Improved rates for identification of partially observed linear
  dynamical systems.
\newblock {\em CoRR}, abs/2011.10006, 2020.

\bibitem[Lju98]{ljung1998system}
Lennart Ljung.
\newblock System identification.
\newblock In {\em Signal analysis and prediction}, pages 163--173. Springer,
  1998.

\bibitem[LL19]{LL2019}
Bruce Lee and Andrew Lamperski.
\newblock Non-asymptotic closed-loop system identification using autoregressive
  processes and hankel model reduction, 2019.

\bibitem[LM19]{lugosi2019sub}
G{\'a}bor Lugosi and Shahar Mendelson.
\newblock Sub-gaussian estimators of the mean of a random vector.
\newblock {\em The annals of statistics}, 47(2):783--794, 2019.

\bibitem[LMC07]{li2007robust}
Qiao Li, Roger~G Mark, and Gari~D Clifford.
\newblock Robust heart rate estimation from multiple asynchronous noisy sources
  using signal quality indices and a kalman filter.
\newblock {\em Physiological measurement}, 29(1):15, 2007.

\bibitem[MB07]{mesot2007switching}
Bertrand Mesot and David Barber.
\newblock Switching linear dynamical systems for noise robust speech
  recognition.
\newblock {\em IEEE Transactions on Audio, Speech, and Language Processing},
  15(6):1850--1858, 2007.

\bibitem[Moi15]{moitra2015super}
Ankur Moitra.
\newblock Super-resolution, extremal functions and the condition number of
  vandermonde matrices.
\newblock In {\em Proceedings of the forty-seventh annual ACM symposium on
  Theory of computing}, pages 821--830, 2015.

\bibitem[MW72]{muller1972analysis}
PC~M{\"u}ller and HI~Weber.
\newblock Analysis and optimization of certain qualities of controllability and
  observability for linear dynamical systems.
\newblock {\em Automatica}, 8(3):237--246, 1972.

\bibitem[OO19]{oymak2019non}
Samet Oymak and Necmiye Ozay.
\newblock Non-asymptotic identification of lti systems from a single
  trajectory.
\newblock In {\em 2019 American control conference (ACC)}, pages 5655--5661.
  IEEE, 2019.

\bibitem[PJL20]{pensia2020robust}
Ankit Pensia, Varun Jog, and Po-Ling Loh.
\newblock Robust regression with covariate filtering: Heavy tails and
  adversarial contamination.
\newblock {\em arXiv preprint arXiv:2009.12976}, 2020.

\bibitem[PSBR20]{prasad2018robust}
Adarsh Prasad, Arun~Sai Suggala, Sivaraman Balakrishnan, and Pradeep Ravikumar.
\newblock Robust estimation via robust gradient estimation.
\newblock {\em Journal of the Royal Statistical Society: Series B (Statistical
  Methodology)}, 82(3):601--627, 2020.

\bibitem[RJR20]{RJR2020}
Paria Rashidinejad, Jiantao Jiao, and Stuart Russell.
\newblock Slip: Learning to predict in unknown dynamical systems with long-term
  memory.
\newblock In H.~Larochelle, M.~Ranzato, R.~Hadsell, M.F. Balcan, and H.~Lin,
  editors, {\em Advances in Neural Information Processing Systems}, volume~33,
  pages 5716--5728. Curran Associates, Inc., 2020.

\bibitem[Row02]{rowell2002state}
Derek Rowell.
\newblock State-space representation of lti systems.
\newblock {\em URL: http://web. mit. edu/2.14/www/Handouts/StateSpace. pdf},
  pages 1--18, 2002.

\bibitem[RY20a]{raghavendra2020regression}
Prasad Raghavendra and Morris Yau.
\newblock List decodable learning via sum of squares.
\newblock In {\em Proceedings of the Fourteenth Annual ACM-SIAM Symposium on
  Discrete Algorithms}, pages 161--180. SIAM, 2020.

\bibitem[RY20b]{raghavendra2020list}
Prasad Raghavendra and Morris Yau.
\newblock List decodable subspace recovery.
\newblock In {\em Conference on Learning Theory}, pages 3206--3226. PMLR, 2020.

\bibitem[SBR19]{simchowitz2019learning}
Max Simchowitz, Ross Boczar, and Benjamin Recht.
\newblock Learning linear dynamical systems with semi-parametric least squares.
\newblock In {\em Conference on Learning Theory}, pages 2714--2802. PMLR, 2019.

\bibitem[SBTR12]{shah2012linear}
Parikshit Shah, Badri~Narayan Bhaskar, Gongguo Tang, and Benjamin Recht.
\newblock Linear system identification via atomic norm regularization.
\newblock In {\em 2012 IEEE 51st IEEE conference on decision and control
  (CDC)}, pages 6265--6270. IEEE, 2012.

\bibitem[Sch09]{schiff2009kalman}
Steven~J Schiff.
\newblock Kalman meets neuron: the emerging intersection of control theory with
  neuroscience.
\newblock In {\em 2009 annual international conference of the IEEE engineering
  in medicine and biology society}, pages 3318--3321. IEEE, 2009.

\bibitem[SMT{\etalchar{+}}18a]{simchowitz2018learning}
Max Simchowitz, Horia Mania, Stephen Tu, Michael~I Jordan, and Benjamin Recht.
\newblock Learning without mixing: Towards a sharp analysis of linear system
  identification.
\newblock In {\em Conference On Learning Theory}, pages 439--473. PMLR, 2018.

\bibitem[SMT{\etalchar{+}}18b]{simchowitz2018}
Max Simchowitz, Horia Mania, Stephen Tu, Michael~I. Jordan, and Benjamin Recht.
\newblock Learning without mixing: Towards {A} sharp analysis of linear system
  identification.
\newblock {\em CoRR}, abs/1802.08334, 2018.

\bibitem[SOF22]{SOF2022}
Yue Sun, Samet Oymak, and Maryam Fazel.
\newblock System identification via nuclear norm regularization, 2022.

\bibitem[SPL05]{spinelli2005}
W.~Spinelli, L.~Piroddi, and M.~Lovera.
\newblock On the role of prefiltering in nonlinear system identification.
\newblock {\em IEEE Transactions on Automatic Control}, 50(10):1597--1602,
  2005.

\bibitem[SR19]{sarkar2019near}
Tuhin Sarkar and Alexander Rakhlin.
\newblock Near optimal finite time identification of arbitrary linear dynamical
  systems.
\newblock In {\em International Conference on Machine Learning}, pages
  5610--5618. PMLR, 2019.

\bibitem[SRD19]{sarkar2019nonparametric}
Tuhin Sarkar, Alexander Rakhlin, and Munther~A Dahleh.
\newblock Nonparametric finite time lti system identification.
\newblock {\em arXiv preprint arXiv:1902.01848}, 2019.

\bibitem[SRD22]{SRD2021}
Tuhin Sarkar, Alexander Rakhlin, and Munther~A. Dahleh.
\newblock Finite time lti system identification.
\newblock {\em J. Mach. Learn. Res.}, 22(1), jul 2022.

\bibitem[TP19]{tsiamis2019finite}
Anastasios Tsiamis and George~J Pappas.
\newblock Finite sample analysis of stochastic system identification.
\newblock In {\em 2019 IEEE 58th Conference on Decision and Control (CDC)},
  pages 3648--3654. IEEE, 2019.

\bibitem[TZMP22]{TZMP2022}
Anastasios Tsiamis, Ingvar Ziemann, Nikolai Matni, and George~J. Pappas.
\newblock Statistical learning theory for control: A finite sample perspective,
  2022.

\bibitem[Zha11]{zhang2011}
Yong Zhang.
\newblock Unbiased identification of a class of multi-input single-output
  systems with correlated disturbances using bias compensation methods.
\newblock {\em Mathematical and Computer Modelling}, 53(9):1810--1819, 2011.

\bibitem[ZJS20]{zhu2020robust}
Banghua Zhu, Jiantao Jiao, and Jacob Steinhardt.
\newblock Robust estimation via generalized quasi-gradients.
\newblock {\em arXiv preprint arXiv:2005.14073}, 2020.

\bibitem[ZL21]{zl2021}
Yang Zheng and Na~Li.
\newblock Non-asymptotic identification of linear dynamical systems using
  multiple trajectories.
\newblock {\em {IEEE} Control Systems Letters}, 5(5):1693--1698, nov 2021.

\end{thebibliography}
